\documentclass[12pt]{article}
\usepackage[russian, english]{babel}
\usepackage{epsfig}
\usepackage{amssymb,amsmath,amsfonts,soul,amsthm,enumerate}
\usepackage{amsfonts}
\usepackage{amsmath}
\usepackage{graphicx}
\usepackage{amsthm}
\usepackage[cp1251]{inputenc}
\usepackage[T2A]{fontenc}
\usepackage{mathrsfs}

\usepackage{rotating}
\newtheorem{theorem}{Theorem}[section]

\newtheorem{remark}{Remark}[section]
\newtheorem{definition}{Definition}[section]

\numberwithin{equation}{section}

\begin{document}

\centerline{\bf Criteria for boundedness of a class of integral operators}
\centerline{\bf   from $L_p$ to $L_q$ for $1<q<p<\infty$}

\vspace{5mm}
\centerline{R. Oinarov$^{a,}$\footnote{Corresponding author: o\_ryskul@mail.ru}, A. Temirkhanova$^{a}$, A. Kalybay$^b$}
\begin{center}
$^a$L.N. Gumilyov Eurasian National University, 2 Satpaev Str.,  Astana 101008, Kazakhstan  \\
$^b$KIMEP University, 4 Abay Ave., Almaty 050010, Kazakhstan
\end{center}

\begin{abstract} In the paper, we consider integral operators with non-negative kernels satisfying conditions, which are less restrictive than conditions studied earlier. We establish criteria for the boundedness of these operators in Lebesgue spaces.
\end{abstract}

\vspace{2mm}


\section{Introduction}

Let  $I=(0, \infty)$, $1<p,q<\infty$ and $p'=\frac{p}{p-1}$. Suppose that $u$ and $v$ are positive functions locally summable on $I$ such that $u\in L_q^{loc} (I)$ and $v\in L_{p'}^{loc} (I)$.

We consider the problem of boundedness from $L_p (I)$ to $L_q (I)$, $1<q<p<\infty$, of the following integral operators
\begin{equation}\label{1.1}
 \mathcal{K}^{+} f(x)= \int\limits_0^x u(x) K(x,s) v(s) f(s) ds, ~x>0,
\end{equation}
\begin{equation}\label{1.2}
 \mathcal{K}^{-} g(s)= \int\limits_s^\infty v(s) K(x,s) u(x) g(x) dx, ~s>0,
\end{equation}
with a kernel $K(x,s)\geq 0$, $x\geq s>0$, i.e., the validity of the inequalities
\begin{equation}\label{1.3+}
\|\mathcal{K}^{+} f\|_q \leq C \|f\|_p, ~ f \in L_p (I),
\end{equation}
\begin{equation}\label{1.3-}
\|\mathcal{K}^{-} f\|_q \leq C \|f\|_p, ~ f \in L_p (I),
\end{equation}
where  $\|\cdot\|_p $ is the standard norm of the space $L_p (I)$.

Taking $K(\cdot,\cdot)\equiv 1$ in \eqref{1.1} and \eqref{1.2}, we respectively get the weighted Hardy operator and the dual weighted Hardy operator, for which inequalities \eqref{1.3+} and \eqref{1.3-} were characterized for all possible parameters $0 <q \leq\infty$ and $1\leq p\leq \infty$ (see, e.g., \cite{AM,B,K,Maz1,Maz2,M,S,SS}).

At the end of the last century, research began on Hardy operators with kernels of forms \eqref{1.1} and \eqref{1.2}.
However, since the problem of characterizing inequalities \eqref{1.3+} and \eqref{1.3-} with an arbitrary measurable kernel $K(\cdot,\cdot)\geq 0$ is open, specific kernels were considered. In papers \cite{S1,S2,S3}, necessary and sufficient conditions for the validity of inequalities \eqref{1.3+} and \eqref{1.3-} were found for the kernel $K(x,s)=(x-s)^\alpha$, $\alpha>0$. These results were successfully used to establish the discreteness of the spectrum of certain higher-order ``polar'' operators (see \cite{S1}).

The problem arose of selecting a class of kernels, which would cover the kernels of known operators.
At the beginning of the 1990s, in the works \cite{BK} and \cite{O1}, a class $\mathcal{O}$ was introduced.  The class $\mathcal{O}$ contains many kernels of different operators and, in particular, kernels of fractional integration operators.
A measurable function $K(x,s)\geq 0$ belongs to the class $\mathcal{O}$ if there exists a number $h\geq 1$ and $h^{-1}\big(K(x,t )+K(t,s)\big)\leq K(x,s)\leq h\big(K(x,t)+K(t,s)\big)$ for $x\geq t\geq s$. Criteria for the boundedness of operators \eqref{1.1} and \eqref{1.2} with a kernel from the class $\mathcal{O}$ have good applications in various problems of analysis, for example, in the oscillatory and spectral theory of higher order differential operators  (see, e.g., \cite{BKO,KOS1,KOS2}). In recent years, operators \eqref{1.1} and \eqref{1.2} with a kernel from the class $\mathcal{O}$  have become the object of many works (see, e.g., \cite{CH,GM,KO,P3,StS1,SYL}). In the last decade, many papers have discussed the weighted estimates of iterated operators, where iterations are taken with respect to different combinations of the Hardy operator, the dual Hardy operator, operators  \eqref{1.1} and \eqref{1.2} with kernels from the class $\mathcal{O}$ (see, e.g., \cite{GS,GMPTU,GM,KO,P3,StS1,StS2}). However, the  kernels of iterations of operators \eqref{1.1} and \eqref{1.2} with a kernel from the class $\mathcal{O}$ do not belong to the class $\mathcal{O}$.

In the paper \cite{O3}, the  classes $\mathcal{O}_n^{+}$ and $\mathcal{O}_n^{-}$, $n\geq 0$, of functions $K(x,s)\geq 0$, $x\geq s>0$, were introduced. These classes are wider than the class $\mathcal{O}$. Necessary and sufficient conditions of boundedness of operators  \eqref{1.1} and \eqref{1.2} from $L_p (I)$ to $L_q (I)$ were established for $1<p\leq q<\infty$ when their kernels $K(\cdot, \cdot)$ belong to the class $\mathcal{O}_n^{+}\bigcup\mathcal{O}_n^{-}$, $n\geq 1$. These classes form extension systems $\mathcal{O}_m^{\pm}\subset \mathcal{O}_n^{\pm}$, $n\geq m\geq 0$.
The work \cite{A} shows that the iteration kernel of two operators \eqref{1.1} and \eqref{1.2} from the classes $\mathcal{O}_n^{\pm}$ and $\mathcal{O}_m^ {\pm}$ belong to the class $\mathcal{O}_{n+m+1}^{\pm}$, i.e., the class system $\{\mathcal{O}_n^{\pm},~n\geq 0\}$ is closed with respect to iteration. This property makes it possible to study various problems where iterations of integral operators are involved, in particular, the above problem of weight estimates for quasilinear operators with iteration of operators \eqref{1.1} and \eqref{1.2}.

The boundedness from $L_p (I)$ to $L_q (I)$ of operators \eqref{1.1} and \eqref{1.2} with kernels from  $\mathcal{O}_n^{+}$ or $\mathcal{O}_n^{-}$, $n>1$,  has not been established for  $1<q<p<\infty$. In the paper \cite{AOP}, under some assumptions on kernels criteria of boundedness of operators \eqref{1.1} and \eqref{1.2} from $L_p (I)$ to $L_q(I)$ for $1<q<p<\infty$ were obtained when their kernels belong to the classes $\mathcal{O}_1^{+}$ or $\mathcal{O}_1^{-}$.

The main aim of this paper is to establish criteria for the boundedness of operators \eqref{1.1} and \eqref{1.2} from $L_p (I)$ to $L_q(I)$ for $1<q<p<\infty$, when their kernels belong to the classes $O_n^{\pm}$, $n\geq 2$. To achieve this aim, we need to find the validity of inequalities   \eqref{1.3+} and \eqref{1.3-} for $1<q<p<\infty$ for more general operators than operators \eqref{1.1} and \eqref{1.2}, when their kernels belong to the classes  $\mathcal{O}_n^{\pm}$.
The specific form of these operators and in which spaces they are considered will be given later. Now, let us define the classes $\mathcal{O}_n^{\pm}$, $n\geq 0$.
Let $\Omega = \{(x,s)\in I\times I: x\geq s>0\}$.

\begin{definition} A measurable function $K(\cdot,\cdot)\equiv K_0(\cdot,\cdot)\geq 0$ defined on the set $\Omega $ belongs to the class $\mathcal{O}_0^{+}$ ($\mathcal{O}_0^{-}$) if  $ K_0(x,s)\equiv r(s)\geq 0$  ($ K_0(x,s)\equiv \bar{r}(x)\geq 0$).
\end{definition}
\begin{definition} A measurable function $K(\cdot,\cdot)\equiv K_n(\cdot,\cdot)\geq 0$ defined on the set $\Omega $ belongs to the class $\mathcal{O}_n^{+}$, $n\geq 1$, if the function $K(\cdot,\cdot)\geq 0$ does not decrease in the first argument and there exists a non-negative functions  $K_{n,i}(\cdot,\cdot)$, $i=0,1,...,n-1$, and $K_i(\cdot,\cdot)$, $i=0,1,...,n-1$, measurable on $\Omega$ and a number $h_n \geq 1$ such that $K_i(\cdot,\cdot)\in \mathcal{O}_i^{+}$, $i=0,1,...,n-1$, and
\begin{equation}\label{1.4}
h_n^{-1} \sum\limits_{i=0}^{n}K_{n,i}(x,t)K_i(t,s) \leq K_n(x,s) \leq  h_n \sum\limits_{i=0}^{n}K_{n,i}(x,t)K_i(t,s)
\end{equation}
for all $x,t,s: x \geq t \geq s>0$, where $K_{n,n}(\cdot,\cdot)\equiv 1$,  $n\geq 1$.
\end{definition}
\begin{definition} A measurable function $K(\cdot ,\cdot )\equiv K_n(\cdot ,\cdot )\geq 0$ defined on the set $\Omega $ belongs to the class $\mathcal{O}_n^{-}$, $n\geq 1$, if the function $K(\cdot ,\cdot )\geq 0$ does not increase in the second argument and there exists a non-negative functions $K_{i,n}(\cdot,\cdot)$, $i=0,1,...,n-1$, and $K_i(\cdot,\cdot)$, $i=0,1,...,n-1$,  measurable on $\Omega$ and a number $\bar{h}_n \geq 1$ such that $K_i(\cdot,\cdot)\in \mathcal{O}_i^{-}$, $i=0,1,...,n-1$, and
\begin{equation}\label{1.5}
\bar{h}_n^{-1} \sum\limits_{i=0}^{n}K_{i}(x,t)K_{i,n}(t,s) \leq K_n(x,s) \leq  \bar{h}_n \sum\limits_{i=0}^{n}K_{i}(x,t)K_{i,n}(t,s) \end{equation}
for all $x,t,s: x \geq t \geq s>0$, where $K_{n,n}(\cdot,\cdot)\equiv 1$,  $n\geq 1$.
\end{definition}

From \eqref{1.4} for $n=1$, i.e., when $K(\cdot,\cdot)\equiv K_1(\cdot,\cdot)\in \mathcal{O}^+_1$, we have
 \begin{equation}\label{1.4.1}
 h_1^{-1}\left(K_{1,0}(x,t)r_1(s)+K_1(t,s)\right)\leq K_1(x,s)\leq h_1\left(K_{1,0}(x,t)r_1(s)+K_1(t,s)\right)
 \end{equation}
for all $x,t,s: x \geq t \geq s>0$.

From \eqref{1.5} for $n=1$, i.e., when $K(\cdot,\cdot)\equiv K_1(\cdot,\cdot)\in \mathcal{O}^-_1$, we have
$$
 \bar{h}_1^{-1}\left(K_{1}(x,t)+\bar{r}_1(x)K_{0,1}(t,s)\right)\leq K_1(x,s)\leq \bar{h}_1\left(K_{1}(x,t)+\bar{r}_1(x)K_{0,1}(t,s)\right)
$$
for all $x,t,s: x \geq t \geq s>0$.

We can similarly write the analogues of relations \eqref{1.4.1}  and \eqref{1.5} for $n=2$.

\begin{remark}  There are many examples of kernels $K(\cdot, \cdot)\geq 0$ that belong to the classes $\mathcal{O}_n^{\pm}$, $n\geq 1$, some of them are given in \cite{O3} and \cite{O4}. Moreover, in \cite{A}, it is proved that the kernels of the superpositions  of two operators \eqref{1.1} and \eqref{1.2} with the kernels $K(\cdot ,\cdot )\in \mathcal{O}_n^{\pm}$ and $K(\cdot ,\cdot )\in \mathcal{O}_m^{\pm}$ belong to the class $\mathcal{O}_{n+m+1}^{\pm}$.
\end{remark}
\begin{remark} \label{r2.0} As shown in \cite{O3}, without loss of generality, the functions $K_{n,i}(\cdot,\cdot)$ and  $K_{i,n}(\cdot,\cdot)$, $i=0,1,...,n-1$, $n\geq 1$, can be considered as non-decreasing in the first argument and non-increasing in the second argument. It is also shown that for any $n,j,i:n\geq j\geq i\geq 0$ we have that $K_{n,j}(x,t)K_{j,i}(t,s)\leq K_{n,i}(x,s)$ and $K_{i,j}(x,t)K_{j,n}(t,s)\leq K_{i,n}(x,s)$ hold for $x\geq t\geq s>0$.
\end{remark}

\begin{remark}\label{r2.1} Let $\varphi$ be a non-decreasing function on $I$. Without loss of generality, we will assume that the function $\varphi$ is right-continuous, replacing $\varphi(t)$ by $\varphi(t+0)$ if necessary. Then it is known (see \cite[Chapter 12]{R}) that there exists a Borel measure $\eta$ such that $\varphi(t)=\int\limits_{[0,t]}d\eta(x)$. If $\varphi(0)=0$, then   $\varphi(t)=\int\limits_{(0,t]}d\eta(x)$.
\end{remark}
\begin{remark}\label{r2.2} Let $ \psi$ be a non-increasing function on $I$. Without loss of generality, we will assume that the function $\psi$ is left-continuous, replacing $\psi(t)$ by $\psi(t-0)$ if necessary. Then it is known (see \cite[Chapter 12]{R}) that there exists a Borel measure $\theta$ such that $\psi(t)=\int\limits_{[t,\infty]}d\theta(x)$. If $\lim\limits_{t\to\infty}\psi(t)=\psi(\infty)=0$, then   $\psi(t)=\int\limits_{[t,\infty)}d\theta(x)$.
\end{remark}
Below, in the statements of Theorems, instead of $d\eta(\cdot)$ and $d\theta(\cdot)$, we will respectively write $d\varphi(\cdot)$ and $d(-\psi(\cdot))$, which indicates that the measures are generated by the functions $\varphi$ and $\psi$.

The paper is organized as follows: In Section 2, we establish the boundedness of operators \eqref{1.1} and \eqref{1.2} from $L_p (I)$ to $L_q(I)$ in the case $1<q<p<\infty$ when their kernels belong to the classes $\mathcal{O}_1^{\pm}$. In Section 3, we obtain similar results for operators \eqref{1.1} and \eqref{1.2} when their kernels belong to the classes $\mathcal{O}_n^{\pm}$, $n\geq 2$.

Throughout the paper, the symbol $A\ll B$ means that $A\leq CB$ with some constant $C>0$. The symbol $A\approx B$ means that $A\ll B\ll A$. Moreover, $\chi_{(a,b)}$ stands for the characteristic function of the interval $(a,b)\subset \mathbb{R}$ and $\mathbb{Z}$ is the set of integers.

\section{Integral operators with kernels from the class $\mathcal{O}_1^{\pm}$}

Let $\Psi$ be the set of nonnegative Borel measures  on $I$. Suppose that $\mu\in \Psi$.  We consider the integral operators
\begin{equation}\label{eq2.1}
\mathbf{K^+}f(x)=\int\limits_0^x K(x, s)v(s)f(s)ds, ~x>0,
\end{equation}
\begin{equation}\label{eq2.2}
\mathbf{K^-}g(s)=\int\limits_{[s,\infty]} K(x, s)v(s)g(x)d\mu(x),~s>0.
\end{equation}
Let  $L_{q, \mu}(I)$ be a space of $\mu$-measurable functions on $I$, for which
$$\|g\|_{q, \mu}=\left(\int\limits_{[0,\infty]}|g(x)|^qd\mu(x)\right)^{\frac{1}{q}}<\infty.$$

Operator \eqref{eq2.1} acts from $L_{p}(I)$ to $L_{q, \mu}(I)$, and operator \eqref{eq2.2} acts from  $L_{p, \mu}(I)$ to $L_{q}(I)$. It is easy to see that operator \eqref{eq2.2} is dual to operator \eqref{eq2.1} from
$L_{q',\mu}(I)$ to $L_{p'}(I)$, and \eqref{eq2.1} is dual to operator \eqref{eq2.2} from $L_{q'}(I)$ to $L_{p',\mu}(I)$ with respect to the linear form $\int\limits_0^\infty f(x)g(x)d\mu(x)$.

\begin{theorem}\label{2.1} Let $1<q<p<\infty$,  $\mu\in \Psi$ and $K(\cdot,\cdot)\equiv K_1(\cdot, \cdot)\in \mathcal{O}_1^{+} $. Then operator \eqref{eq2.1} is bounded from  $L_p(I)$ to $L_{q, \mu} (I)$ if and only if $B^+_1=\max\{B^+_{1,0}, B^+_{1,1}\}<\infty$. Moreover, $\|\mathbf{K}^+\|_{p\rightarrow q}\approx B^+_1$,  where $\|\mathbf{K}^+\|_{p\rightarrow q}$ is the norm of operator \eqref{eq2.1} from $L_p(I)$ to $L_{q, \mu} (I)$, where
$$
B^+_{1,0}=\left(\int\limits_0^\infty\left(\int\limits_{[z, \infty]} K_{1, 0}^q (x, z) d\mu(x)\right)^{\frac{p}{p-q}} \left(\int\limits_0^z r_1^{p'}(s)v^{p'}(s)ds\right)^{\frac{p(q-1)}{p-q}}r_1^{p'}(z)v^{p'}(z)dz\right)^{\frac{p-q}{pq}},
$$
$$
B^+_{1,1}=\left(\int\limits_{(0,\infty]}\left(\mu\left([x, \infty]\right)\right)^{\frac{p}{p-q}} \left(\int\limits_0^x K_1^{p'}(x, s)v^{p'}(s)ds\right)^{\frac{p(q-1)}{p-q}} d\left(\int\limits_0^x K_1^{p'}(x, t)v^{p'}(t)dt\right)\right)^{\frac{p-q}{pq}}.
$$
\end{theorem}
\begin{proof}  {\it  Necessity.} Let operator \eqref{eq2.1} be bounded from $L_p(I)$ to $L_{q, \mu}(I)$, $1<q<p<\infty$, with the norm $\|\mathbf{K}^+\|_{p\rightarrow q}<\infty$. Let $0\leq f\in L_p(I)$ have compact support. Using the relation $K_1(x, t)\gg   K_{1, 0} (x, s) r_1(t)$ for $x\geq s\geq t>0$, which follows from (\ref{1.4.1}), by Fubini's theorem, we have
$$
J=\int\limits_{[0,\infty]}\left(\mathbf{K}^+f(x)\right)^qd\mu(x)\approx\int\limits_{[0,\infty]}\int\limits_0^x K_1(x, s)v(s)f(s)\left(\int\limits_0^s K_1(x, t)v(t)f(t)dt\right)^{q-1}ds d\mu(x)
$$
$$
\gg  \int\limits_{[0,\infty]}\int\limits_0^x K_{1, 0}^q (x, s)r_1(s)v(s)f(s)\left(\int\limits_0^s r_1(t)v(t)f(t)dt\right)^{q-1} ds d\mu(x)
$$
\begin{equation}\label{F2.1}
=\int\limits_0^\infty r_1(s)v(s)f(s)\left(\int\limits_0^s r_1(t)v(t)f(t)dt\right)^{q-1}\int\limits_{[s, \infty]} K_{1, 0}^q(x, s)d\mu(x) ds.
\end{equation}
Since  the function $\int\limits_{[s, \infty]} K_{1, 0}^q(x, s)d\mu(x)$ does not increase on $I$,  then, by Remark \ref{r2.2}, there exists a Borel measure $\theta_1$  such that $\int\limits_{[s, \infty]} K_{1, 0}^q(x, s)d\mu(x)=\int\limits_{[s, \infty]}d\theta_1(t)$. Replacing this relation into \eqref{F2.1} and using Fubini's theorem, we deduce
$$
J=\int\limits_0^\infty r_1(s)v(s)f(s)\left(\int\limits_0^s r_1(t)v(t)f(t)dt\right)^{q-1}\int\limits_{[s, \infty]} d\theta_1(z) ds
$$
$$
=\int\limits_{[0,\infty]}\int\limits_0^z r_1(s)v(s)f(s)\left(\int\limits_0^s r_1(t)v(t)f(t)dt\right)^{q-1}ds\,d\theta_1(z)
$$
$$
\approx\int\limits_{[0,\infty]}\left(\int\limits_0^z r_1(t)v(t)f(t)dt\right)^{q} d\theta_1(z).
$$
Then
$$
\left(\int\limits_{[0,\infty]}\left(\int\limits_0^z r_1(t)v(t)f(t)dt\right)^{q} d\theta_1(z)\right)^{\frac{1}{q}}\ll \|\mathbf{K^+}\|_{p\rightarrow q}\left(\int\limits_0^\infty f^p(t)dt\right)^{\frac{1}{p}}.
$$
Hence, by Theorem 2 from \cite[\S 1.3, 1.3.3]{Maz2}, we get
\begin{equation}\label{eq2.3}
B^+_{1, 0}\ll \|\mathbf{K^+}\|_{p\rightarrow q}.
\end{equation}
From the boundedness of operator \eqref{eq2.1} from $L_p(I)$ to $L_{q, \mu} (I)$ it follows the boundedness of dual operator \eqref{eq2.2} from $L_{q', \mu} (I)$ to $L_{p'}(I)$, i.e., we have
\begin{equation}\label{eq2.4}
\left(\int\limits_0^\infty\left(\int\limits_{[s, \infty]} K_1(x, s) v(s)g(x)d\mu(x)\right)^{p'}ds\right)^{\frac{1}{p'}} \ll \|\mathbf{K^+}\|_{p\rightarrow q}\left(\int\limits_{[0,\infty]} |g(t)|^{q'}d\mu(t)\right)^{\frac{1}{q'}}
\end{equation}
for $0\leq g\in L_{q', \mu}(I)$ having compact support.
The non-decrease in the first argument of the function $K_1(\cdot, \cdot)\in \mathcal{O}^+_1$ gives that $K_1(t, s)\geq K_1(x, s)$ for $t\geq x\geq s>0$. This fact and Lemma 2 from \cite{P2} imply that
$$
J^-=\int\limits_0^\infty(\mathbf{K}^-g(s))^{p'}ds\approx\int\limits_0^\infty\int\limits_{[s, \infty]} K_1(x, s)v^{p'}(s)g(x)\left(\int\limits_{[x, \infty]} K_1(t, s)g(t)d\mu(t)\right)^{p'-1}d\mu(x)ds
$$
$$
\geq\int\limits_0^\infty v^{p'}(s)\int\limits_{[s, \infty]} K_1^{p'}(x, s)g(x)\left(\int\limits_{[x, \infty]}g(t)d\mu(t)\right)^{p'-1}d\mu(x)ds
$$
\begin{equation}\label{F2.2}
=\int\limits_{[0,\infty]} g(x)\left(\int\limits_{[x, \infty]} g(t)d\mu(t)\right)^{p'-1}\int\limits_0^x K_1^{p'}(x, s)v^{p'}(s)ds\,d\mu(x).
\end{equation}
The function $\int\limits_0^x K_1^{p'}(x, s)v^{p'}(s)ds $ does not decrease on $I$ and $\lim\limits_{x\to 0}\int\limits_0^x K_1^{p'}(x, s)v^{p'}(s)ds =0$ due to \eqref{F2.2}. Therefore, by Remark  \ref{r2.1}, there exists a Borel measure $\eta_1$ such that $\int\limits_0^x K_1^{p'}(x, s)v^{p'}(s)ds =\int\limits_{(0,x]}d\eta_1(t)$. Replacing this relation into \eqref{F2.2} and using Fubini's theorem, we obtain
$$
J^-=\int\limits_{[0,\infty]} g(x)\left(\int\limits_{[x, \infty]}g(t)d\mu(t)\right)^{p'-1}\int\limits_{(0, x]} d\eta_1(z)d\mu(x)
\approx\int\limits_{(0,\infty]}\left(\int\limits_{[z, \infty]} g(t)d\mu(t)\right)^{p'} d\eta_1(z).
$$
The latter and \eqref{eq2.4} give that
$$
\left(\int\limits_{(0,\infty]}\left(\int\limits_{[z, \infty]} g(t)d\mu(t)\right)^{p'} d\eta_1(z)\right)^{\frac{1}{p'}}\ll \|\mathbf{K^+}\|_{p\rightarrow q} \left(\int\limits_{[0,\infty]}|g(t)|^{q'}d\mu(t)\right)^{\frac{1}{q'}}.
$$
Then, on the basis of Theorem 3 from \cite{P2}, we have $B^+_{1, 1}\ll \|\mathbf{K^+}\|_{p\rightarrow q}$, which, together with  \eqref{eq2.3}, implies that
\begin{equation}\label{eq2.5}
\|\mathbf{K^+}\|_{p\rightarrow q}\gg  B^+_1.
\end{equation}

{\it Sufficiency.} Let $B^+_1<\infty$ and $K(\cdot, \cdot)\equiv K_1(\cdot, \cdot)\in \mathcal{O}_1^+$. Let $0\leq f\in L_p(I)$ have compact support. Hence, the function $F(x)=\int\limits_0^x K_1(x, s)v(s)f(s)ds$ is non-negative, non-decreasing and turns to zero at $x\rightarrow 0^+$. Let $\alpha_f=\inf\{x>0: F(x)>0\}$. Then $F(x)>0$ for $x>\alpha_f$. Then, without loss of generality, we assume that $\alpha_f=0$.

We define the function (see \cite{O3})
$$
k(x)=\max\{k\in \mathbb{Z}: (h_1+1)^k\leq F(x),  x\in I\},
$$
where $h_1$ from (\ref{1.4.1}). The function $k(x)$ is non-decreasing and
$$
(h_1+1)^{k(x)}\leq F(x)\leq (h_1+1)^{k(x)+1}, ~~ \forall x\in I.
$$
Let $\{k_i\}$, $(k_i<k_{i+1})$, be the domain of the function $k: I\rightarrow \mathbb{Z}$ and $I_i$ is the preimage of the point $k_i\in \mathbb{Z}$, i.e., $I_i=\{x\in I: k(x)=k_i\}$. Then
\begin{equation}\label{eq2.6}
I=\bigcup\limits_i I_i.
\end{equation}
Assume that $x_i=\inf I_i$. It is obvious that $x_i<x_{i+1}$ and $I_i=[x_i, x_{i+1})$. Therefore,
\begin{equation}\label{eq2.7}
(h_1+1)^{k_i}\leq F(x)< (h_1+1)^{k_i+1}~~\text{for}~~x_i\leq x<x_{i+1}.
\end{equation}
Since $k_i<k_{i+1}$, then $k_i-1\geq k_{i-1}\geq k_{i-2}+1$ and
$$
(h_1+1)^{k_i-1}=(h_1+1)^{k_i}-h_1(h_1+1)^{k_i-1}\leq (h_1+1)^{k_i}-h_1(h_1+1)^{k_{i-2}+1}\leq F(x_i)-h_1 F(x_{i-2})
$$
$$
\leq\int\limits_{x_{i-2}}^{x_i} K_1(x_i, s)v(s)f(s)ds+\int\limits_0^{x_{i-2}}[K_1(x_i, s)-h_1K_1(x_{i-2}, s)]v(s)f(s)ds
$$
\begin{equation}\label{eq2.8}
\leq\int\limits_{x_{i-2}}^{x_i} K_1(x_i, s)v(s)f(s)ds+h_1 K_{1, 0}(x_i, x_{i-2})\int\limits_0^{x_{i-2}}r_1(s)v(s)f(s)ds.
\end{equation}
In the last relation, we have applied the right-hand side of (\ref{1.4.1}) to $K_1(x_i, s)-h_1K_1(x_{i-2}, s)$. In view of \eqref{eq2.6}, \eqref{eq2.7}  and \eqref{eq2.8}, assuming that $\bar{I}_i=[x_i,x_{i+1}]$, we get
$$
J=\int\limits_{[0,\infty]}(\mathbf{K^+}f(x))^qd\mu(x)\ll\sum\limits_i\int\limits_{\bar{I}_i}F^q(x)d\mu(x)\ll \sum\limits_i(h_i+1)^{q(k_i-1)}\mu(\bar{I}_i)
$$
$$
\ll \sum\limits_i\mu(\bar{I}_i)\left(\int\limits_{x_{i-2}}^{x_i}K_1(x_i, s)v(s)f(s)ds\right)^q+\sum\limits_i\mu(\bar{I}_i)K_{1, 0}^q(x_i, x_{i-2})\left(\int\limits_0^{x_{i-2}}r_1(s)v(s)f(s)ds\right)^q
$$
$$
\leq \sum\limits_i\mu(\bar{I}_i)\left(\int\limits_{x_{i-2}}^{x_i}K_1(x_i, s)v(s)f(s)ds\right)^q+\sum\limits_i\int\limits_{\bar{I}_i}K_{1, 0}^q(x, x_{i-2})d\mu(x)\left(\int\limits_0^{x_{i-2}}r_1(s)v(s)f(s)ds\right)^q
$$
\begin{equation}\label{eq2.9}
=J_{1, 1}+J_{1, 0}.
\end{equation}
We separately estimate $J_{1, 0}$ and $J_{1, 1}$. To estimate $J_{1, 1}$ we use H\"{o}lder's inequality in the integral, then in the sum with the parameters $\frac{p}{q}$ and   $\frac{p}{p-q}$:
$$
J_{1, 1}\leq \sum\limits_i\mu(\bar{I}_i)\left(\int\limits_{x_{i-2}}^{x_i}K_1^{p'}(x_i, s)v^{p'}(s)ds\right)^{\frac{q}{p'}}\left(\int\limits_{x_{i-2}}^{x_i}|f(t)|^pdt\right)^{\frac{q}{p}}
$$
\begin{equation}\label{eq2.10}
\leq\left(\sum\limits_i\mu^{\frac{p}{p-q}}(\bar{I}_i)\left(\int\limits_{x_{i-2}}^{x_i}K_1^{p'}(x_i, s)v^{p'}(s)ds\right)^{\frac{q(p-1)}{p-q}}\right)^{\frac{p-q}{p}}\left(\sum\limits_i\int\limits_{x_{i-2}}^{x_i}|f(t)|^pdt\right)^{\frac{q}{p}}\ll \hat{B}_{1, 1}\|f\|_p^q,
\end{equation}
where
$$
\hat{B}_{1, 1}=\left(\sum\limits_i\mu^{\frac{p}{p-q}}(\bar{I}_i)\left(\int\limits_{x_{i-2}}^{x_i}K_1^{p'}(x_i, s)v^{p'}(s)ds\right)^{\frac{q(p-1)}{p-q}}\right)^{\frac{p-q}{p}}.
$$
Using Lemma 2 from \cite{P2} and the relation
$\mu^{\frac{p}{p-q}}(\bar{I}_i) \approx \int\limits_{[x_i,x_{i+1}]}\mu^{\frac{q}{p-q}}([x,x_{i+1}])d\mu(x)$,
then  using as above the relation $\int\limits_0^{x}K_1^{p'}(x, s)v^{p'}(s)ds=\int\limits_{(0,x]}d\eta_1(z)$ and Fubini's theorem, we obtain
$$
\hat{B}_{1, 1}\ll \left(\sum\limits_i\int\limits_{[x_i,x_{i+1}]}\mu^{\frac{q}{p-q}}([x,\infty])\left(\int\limits_0^{x}K_1^{p'}(x, s)v^{p'}(s)ds\right)^{\frac{q(p-1)}{p-q}} d\mu(x)\right)^{\frac{p-q}{p}}
$$
$$
\ll \left(\int\limits_{[0,\infty]}(\mu([x, \infty]))^{\frac{q}{p-q}}\left(\int\limits_{(0,x]}d\eta_1(z)\right)^{\frac{q(p-1)}{p-q}} d\mu(x)\right)^{\frac{p-q}{p}}
$$
$$
\approx\left(\int\limits_{[0,\infty]}(\mu([x, \infty]))^{\frac{q}{p-q}}\int\limits_{(0,x]}\left(\int\limits_{(0,s]}d\eta_1(z)\right)^{\frac{p(q-1)}{p-q}}d\eta_1(s) d\mu(x)\right)^{\frac{p-q}{p}}
$$
$$
\approx\left(\int\limits_{(0,\infty]}(\mu([s, \infty]))^{\frac{p}{p-q}}\left(\int\limits_{(0,s]}d\eta_1(z)\right)^{\frac{p(q-1)}{p-q}}d\eta_1(s) \right)^{\frac{p-q}{p}}
$$
$$
=\left(\int\limits_{(0,\infty]}(\mu([x, \infty]))^{\frac{p}{p-q}}\left(\int\limits_0^{x}K_1^{p'}(x, s)v^{p'}(s)ds\right)^{\frac{p(q-1)}{p-q}}d\left(\int\limits_0^{x}K_1^{p'}(x, s)v^{p'}(s)ds\right) \right)^{\frac{p-q}{p}}=(B^+_{1, 1})^q.
$$
The latter and \eqref{eq2.10} give that
\begin{equation}\label{eq2.11}
J_{1, 1}\ll (B^+_{1, 1})^q\|f\|_p^q.
\end{equation}
Now, we estimate $J_{1, 0}$:
$$
J_{1, 0}=\sum\limits_i\int\limits_{\bar{I}_i} K_{1, 0}^q(x, x_{i-2}) d\mu(x)\left(\int\limits_0^{x_{i-2}} r_1(s)v(s)f(s)ds\right)^q=\int\limits_{[0,\infty]}\left(\int\limits_0^t r_1(s)v(s)f(s)ds\right) d\eta(t),
$$
where $d\eta(t)=\sum\limits_i\int\limits_{\bar{I}_i}K_{1, 0}^q(x, x_{i-2}) d\mu(x)\delta(t-x_{i-2})dt$ and $\delta(\cdot)$  is the Dirac delta function. Then, by Theorem 2 from \cite[\S 1.3, 1.3.3]{Maz2}, we have
$$
J_{1, 0}\ll \hat{B}_{1, 0}\left(\int\limits_0^\infty|f(t)|^pdt\right)^{\frac{q}{p}},
$$
where
$$
\hat{B}_{1, 0}=\left(\int\limits_0^\infty (\eta([x, \infty]))^{\frac{p}{p-q}}\left(\int\limits_0^x r_1^{p'}(s)v^{p'}(s)ds\right)^{\frac{p(q-1)}{p-q}} r_1^{p'}(x)v^{p'}(x)dx\right)^{\frac{p-q}{p}}.
$$
Since $\eta([x, \infty])=\sum\limits_{i: x_{i-2}\geq x}\int\limits_{\bar{I}_i} K^q_{1, 0}(t, x_{i-2}) d\mu(t)\ll \int\limits_{[x, \infty]}  K_{1, 0}^q(t, x) d\mu(t)$,
then $\hat{B}_{1, 0}\ll (B^+_{1, 0})^q$. Therefore, $J_{1, 0}\ll (B^+_{1, 0})^q\|f\|_p^q$. The last estimate, together with \eqref{eq2.9} and \eqref{eq2.11}, gives that $J\ll \max\{B^+_{1, 1}, B^+_{1, 0}\}^q\|f\|_p^q$, i.e., operator \eqref{eq2.1} is bounded from $L_p(I)$ to $L_{q, \mu}(I)$ and for its norm the estimate $\|\mathbf{K^+}\|_{p\rightarrow q}\ll  B^+_1$ holds, which and \eqref{eq2.5} give that $\|\mathbf{K^+}\|_{p\rightarrow q}\approx B^+_1$. The proof of Theorem \ref{2.1} is complete.
\end{proof}

We can similarly prove the following statement.

\begin{theorem}\label{2.2} Let $1<q<p<\infty$,  $\mu\in \Psi$ and $K(\cdot,\cdot)\equiv K_1(\cdot, \cdot)\in \mathcal{O}_1^{-} $. Then operator \eqref{eq2.2} is bounded from  $L_{p,\mu}(I)$ to $L_{q} (I)$ if and only if $B^-_1=\max\{B^-_{0,1}, B^-_{1,1}\}<\infty$. Moreover, $\|\mathbf{K}^-\|_{p\rightarrow q}\approx B^-_1$, where  $\|\mathbf{K}^-\|_{p\rightarrow q}$ is the norm of operator \eqref{eq2.2} from $L_{p,\mu}(I)$ to $L_{q} (I)$, where
$$
B^-_{0,1}=\left(\int\limits_{(0,\infty]}\left(\int\limits_{[z, \infty]} \bar{r}^{p'}(t)d\mu(t)\right)^{\frac{q(p-1)}{p-q}} \left(\int\limits_0^z K^q_{0,1}(z,s)v^{q}(s)ds\right)^{\frac{q}{p-q}}d\left(\int\limits_0^z K^q_{0,1}(z,s)v^{q}(s)ds\right)\right)^{\frac{p-q}{pq}},
$$
$$
B^-_{1,1}=\left(\int\limits_0^\infty\left(\int\limits_{[z,\infty]}K^{p'}_1(x,z)d\mu(x)\right)^{\frac{q(p-1)}{p-q}} \left(\int\limits_0^z v^{q}(s)ds\right)^{\frac{q}{p-q}} v^{q}(z)dz\right)^{\frac{p-q}{pq}}.
$$
\end{theorem}

\section{Operators \eqref{1.1} and \eqref{1.2} with kernels from the classes $\mathcal{O}_n^{\pm}$}

In this Section, we establish the boundedness of operators \eqref{1.1} and \eqref{1.2} from $L_{p}(I)$ to $L_{q} (I)$, $1<q<p<\infty$, when their kernels belong to  the class  $\mathcal{O}_n^{\pm}$. It follows from Theorems \ref{2.1} and \ref{2.2} that in relation (\ref{1.4.1}) the function $r_1$ can be taken equal to unity, since it can be combined with the weight functions of operators \eqref{1.1} and \eqref{1.2}.
Therefore, further we suppose that $K_0(\cdot,\cdot)\equiv1$ in (\ref{1.4}), and (\ref{1.4}) has the form
$$
h_n^{-1} \left(K_{n,0}(x,t)+\sum\limits_{i=1}^{n-1}K_{n,i}(x,t)K_i(t,s)+K_n(t,s)  \right )\leq K_n(x,s)
$$
\begin{equation}\label{eq3.1}
\leq  h_n \left(K_{n,0}(x,t)+\sum\limits_{i=1}^{n-1}K_{n,i}(x,t)K_i(t,s)+K_n(t,s)\right), ~~ x\geq t\geq s>0.
\end{equation}

Assume that
\begin{multline*}
B^+_{n, i}=\left(\int\limits_{(0,\infty]}\left(\int\limits_{[z,\infty]} K_{n, i}^q(x, z) d\mu(x)\right)^{\frac{p}{p-q}}\left(\int\limits_0^z K_i^{p'}(z, s)v^{p'}(s)ds\right)^{\frac{p(q-1)}{p-q}}\right.\\
\left.\times d\left(\int\limits_0^z K^{p'}_i (z, t) v^{p'}(t)dt\right)\right)^{\frac{p-q}{pq}},~i=0,1,...,n.
\end{multline*}

\begin{theorem}\label{th3.1} Let $1<q<p<\infty$,  $\mu\in\Psi$ and $K(\cdot,\cdot)\equiv K_n(\cdot, \cdot)\in \mathcal{O}_n^{+}$, $n\geq1$. Then operator \eqref{eq2.1} is bounded from $L_{p}(I)$ to $L_{q,\mu}(I)$ if and only if $B^+_n=\max\limits_{0\leq i\leq n}B^+_{n, i}<\infty$. Moreover, $\|\mathbf{K^+}\|_{p\rightarrow q}\approx B^+_n$, where $\|\mathbf{K^+}\|_{p\rightarrow q}$ is the norm of operator \eqref{eq2.1} from $L_{p}(I)$ to $L_{q,\mu}(I)$.
In addition,
\begin{multline}\label{eq3.2}
B^+_{n, i}\approx
\left(\int\limits_{[0,\infty]}\left(\int\limits_0^z K_i^{p'}(z, s)v^{p'}(s)ds\right)^{\frac{q(p-1)}{p-q}}\left(\int\limits_{[z,\infty]} K_{n, i}^q(x, z)
 d\mu(x)\right)^{\frac{q}{p-q}}\right.\\
 \left.\times d\left(-\int\limits_{[z,\infty]} K_{n, i}^q(x, z) d\mu(x)\right)\right)^{\frac{p-q}{pq}},~i=0,1,...,n.
\end{multline}
\end{theorem}

\begin{remark}\label{r3.0}
Relations of form \eqref{eq3.2} are presented in the statements of all Theorems because they are suitable for the applications listed in Introduction. We only prove relation \eqref{eq3.2}, since the other relations are proved similarly.

\end{remark}

\begin{proof} {\it Necessity.} Let operator \eqref{eq2.1} be bounded from $L_{p}(I)$ to $L_{q,\mu}(I)$, i.e., the inequality
\begin{equation}\label{eq3.15}
\left(\int\limits_{[0,\infty]}\left(\int\limits_0^x K_n(x, s) v(s)f(s)ds\right)^{q}d\mu(x)\right)^{\frac{1}{q}} \leq \|\mathbf{K^+}\|_{p\rightarrow q}\left(\int\limits_0^\infty |f(t)|^{p}dt\right)^{\frac{1}{p}}
\end{equation}
holds for all $0\leq f\in L_p(I)$ having compact support.

Using the relations $K_n(x, t)\gg  K_{n, 0}(x, s)$ for $x\geq s\geq t>0$ and, in particular, the relation $K_n(x, s)\gg   K_{n, 0}(x, s)$ for $x\geq s>0$, which follows from \eqref{eq3.1}, by Lemma 1 from \cite{P2} and  Fubini's theorem, we have
$$
J_n\equiv\int\limits_{[0,\infty]} \left(\int\limits_0^x K_n(x, s)v(s)f(s)ds\right)^qd\mu(x)
$$
$$
\approx\int\limits_{[0,\infty]} \int\limits_0^x K_n(x, s)v(s)f(s)\left(\int\limits_0^s K_n(x, t)v(t)f(t)dt\right)^{q-1}dsd\mu(x)
$$
$$
\gg\int\limits_{[0,\infty]} \int\limits_0^x K_{n, 0}^q(x, s)v(s)f(s)\left(\int\limits_0^sv(t)f(t)dt\right)^{q-1}dsd\mu(x)
$$
\begin{equation}\label{F3.1}
=\int\limits_0^\infty v(s)f(s)\left(\int\limits_0^sv(t)f(t)dt\right)^{q-1}\int\limits_{[s,\infty]} K_{n, 0}^q(x, s)d\mu(x)ds.
\end{equation}
Remark \ref{r2.0} gives that the function $\int\limits_{[s,\infty)} K_{n, 0}^q(x, s)d\mu(x)$ does not increase on $I$.  Hence, by Remark \ref{r2.2}, there exists a Borel measure $\theta_{n,0}$ such that  $\int\limits_{[s,\infty]} K_{n, 0}^q(x, s)d\mu(x)=\int\limits_{[s,\infty]}d\theta_{n,0}(z)$. Replacing this relation into \eqref{F3.1} and using Fubini's theorem, we get
$$
J_n\gg\int\limits_0^\infty v(s)f(s)\left(\int\limits_0^s v(t)f(t)dt\right)^{q-1}\int\limits_{[s, \infty]} d\theta_{n,0}(z)ds
\approx\int\limits_{[0,\infty]}\left(\int\limits_0^z v(t)f(t)dt\right)^q d\theta_{n,0}(z).
$$
Then from \eqref{eq3.15} we obtain
$$
\left(\int\limits_{[0,\infty]}\left(\int\limits_0^z v(t)f(t)dt\right)^{q}d\theta_{n,0}(z)\right)^{\frac{1}{q}} \ll \|\mathbf{K^+}\|_{p\rightarrow q}\left(\int\limits_0^\infty |f(t)|^{p}dt\right)^{\frac{1}{p}}
$$
and, on the basis of Theorem 2 from \cite[\S 1.3, 1.3.3]{Maz2}, we find that
\begin{equation}\label{eq3.16}
B^+_{n, 0}\ll  \|\mathbf{K^+}\|_{p\rightarrow q}.
\end{equation}
From the boundedness of operator \eqref{eq2.1} from $L_{p}(I)$ to $L_{q,\mu}(I)$ it follows the boundedness of dual operator \eqref{eq2.2} from $L_{q',\mu}(I)$ to $L_{p'}(I)$, i.e., we have
\begin{equation}\label{eq3.17}
\left(\int\limits_0^\infty v^{p'}(s)\left(\int\limits_{[s,\infty]} K_n(x, s)g(x)d\mu(x)\right)^{p'}ds\right)^{\frac{1}{p'}}\leq \|\mathbf{K^+}\|_{p\rightarrow q}\left(\int\limits_{[0,\infty]} |g(t)|^{q'}d\mu(t)\right)^{\frac{1}{q'}}
\end{equation}
for $0\leq g\in L_{q'}(I)$ having compact support. Since $K_n(\cdot, \cdot)\in\mathcal{O}_n^+$, then, by the definition, the function $K_n(\cdot, \cdot)$ does not decrease in the first argument. Using the last fact, Lemma 2 from \cite{P2} and Fubini's theorem, we deduce that
$$
J_n^*=\int\limits_0^\infty v^{p'}(s)\left(\int\limits_{[s,\infty]} K_n(x, s)g(x)d\mu(x)\right)^{p'}ds
$$
$$
\approx\int\limits_0^\infty v^{p'}(s)\int\limits_{[s,\infty]} K_n(x, s)g(x) \left(\int\limits_{[x,\infty]} K_n(t, s)g(t)d\mu(t)\right)^{p'-1}d\mu(x)ds
$$
$$
\geq\int\limits_0^\infty v^{p'}(s)\int\limits_{[s,\infty]} K^{p'}_n(x, s)g(x)\left(\int\limits_{[x,\infty]} g(t)d\mu(t)\right)^{p'-1}d\mu(x)ds
$$
\begin{equation}\label{F3.2}
=\int\limits_{[0,\infty]} g(x) \left(\int\limits_{[x,\infty]} g(t)d\mu(t)\right)^{p'-1}\int\limits_0^x K_n^{p'}(x, s)v^{p'}(s)ds\,d\mu(x).
\end{equation}
The function $\int\limits_0^x K_n^{p'}(x, s)v^{p'}(s)ds$ does not decrease on $I$ and $\lim\limits_{x\to 0}\int\limits_0^x K_n^{p'}(x, s)v^{p'}(s)ds=0$ due to \eqref{F3.2}. Therefore,  by Remark \ref{r2.1}, there exists a Borel measure $\eta_n$ such that $\int\limits_0^x K_n^{p'}(x, s)v^{p'}(s)ds=\int\limits_{(0,x]}d\eta_n(z)$. Using this relation in \eqref{F3.2}, Fubini's theorem and Lemma 2 from \cite{P2}, we have
$$
J_n^*\gg\int\limits_{(0,\infty]}\left(\int\limits_{[z,\infty]} g(t)d\mu(t)\right)^{p'} d\eta_n(z).
$$
The latter and \eqref{eq3.17} give that
$$
\left(\int\limits_{(0,\infty]}\left(\int\limits_{[z,\infty])} g(t)d\mu(t)\right)^{p'} d\eta_n(z)\right)^{\frac{1}{p'}}\ll \|\mathbf{K^+}\|_{p\rightarrow q}\left(\int\limits_{[0,\infty]} |g(t)|^{q'}d\mu(t)\right)^{\frac{1}{q'}}.
$$
Then, by Theorem 3 from \cite{P2}, we get
\begin{equation}\label{eq3.18}
B^+_{n, n}\ll  \|\mathbf{K^+}\|_{p\rightarrow q}.
\end{equation}
From \eqref{eq3.1} it follows that
\begin{equation}\label{eq3.19}
K_{n}(x, s)\gg  K_{n, i}(x, t)K_i(t, s)=\chi_{(s, x)}(t) K_{n, i}(x, t)K_i(t, s),~i=1,2,...,n-1,
\end{equation}
for  $x\geq t\geq s>0$.
Let $ \varphi: I\to [0,\infty)$ and $\int\limits_0^\infty \varphi(t)dt=1$. Multiplying both sides of \eqref{eq3.19} by $\varphi(t)$ and integrating with respect to $t\in I$, we obtain
$$
K_n(x, s)\gg  \int\limits_s^x K_{n, i}(x, t)K_i(t, s)\varphi(t)dt,~i=1,2,...,n-1.
$$
Then, using Fubini's theorem and the fact that the function $K_i(\tau, t)$ does not decrease in the first argument, for an arbitrary $i\in \{1,2,...,n-1\}$ we get
$$
\int\limits_{[0,\infty]} \left(\int\limits_0^x K_n(x, s)v(s)f(s)ds\right)^{q}d\mu(x)
$$
$$
=\int\limits_{[0,\infty]} \int\limits_0^x K_n(x, s)v(s)f(s)ds\left(\int\limits_0^x K_n(x, t)v(t)f(t)dt\right)^{q-1}d\mu(x)
$$
$$
\geq \int\limits_{[0,\infty]} \int\limits_0^x \left(\int\limits_s^x K_{n, i}(x, z)K_i(z, s)\varphi(z)dz\right)v(s)f(s)ds
$$
$$
\times\left(\int\limits_0^x \left(\int\limits_t^x K_{n, i}(x, \tau)K_i(\tau, t)\varphi(\tau)d\tau\right)v(t)f(t)dt\right)^{q-1}d\mu(x)
$$
$$
=\int\limits_{[0,\infty]} \int\limits_0^x K_{n, i}(x, z)\int\limits_0^z K_i(z, s)v(s)f(s)ds\varphi(z)dz
$$
$$
\times\left(\int\limits_0^x K_{n, i}(x, \tau)\int\limits_0^\tau K_i(\tau, t)v(t)f(t)dt\varphi(\tau)d\tau\right)^{q-1}d\mu(x)
$$
$$
\geq\int\limits_0^\infty \int\limits_0^z K_i(z, s)v(s)f(s)ds\int\limits_{[z,\infty]} K_{n, i}(x, z)
$$
$$
\times\left(\int\limits_z^x K_{n,i}(x,\tau)\int\limits_0^\tau K_i(\tau, t)v(t)f(t)dt\varphi(\tau)d\tau\right)^{q-1}d\mu(x)\varphi(z)dz
$$
$$
\geq \int\limits_0^\infty \left(\int\limits_0^z K_i(z, s)v(s)f(s)ds\right)^q \int\limits_{[z,\infty]} K_{n, i}(x, z)\left(\int\limits_z^x K_{n, i}(x, \tau)\varphi(\tau)d\tau\right)^{q-1}d\mu(x)\varphi(z)dz.
$$
Based on \eqref{eq3.15}, the latter gives that
$$
\left(\int\limits_o^\infty \left(\int\limits_0^z K_i(z, s)v(s)f(s)ds\right)^q \int\limits_{[z,\infty]} K_{n, i}(x, z)\left(\int\limits_z^x K_{n, i}(x, \tau)\varphi(\tau)d\tau\right)^{q-1}d\mu(x)\varphi(z)dz\right)^{\frac{1}{q}}
$$
\begin{equation}\label{eq3.20}
\ll\|\mathbf{K^+}\|_{p\rightarrow q}\left(\int\limits_0^\infty f^p(t)dt\right)^{\frac{1}{p}},~i=1,2,...,n-1.
\end{equation}

This inequality is the same as inequality \eqref{eq3.15} for $n=i$. Therefore, due to \eqref{eq3.18}, which holds for all $n>1$, we have
$$
\sup\limits_{\varphi}\left(\int\limits_{(0,\infty]}\left(\int\limits_z^\infty \int\limits_{[t,\infty]} K_{n, i}(x, t)\left(\int\limits_t^x K_{n, i}(x, \tau)\varphi(\tau)d\tau\right)^{q-1}d\mu(x)\varphi(t)dt\right)^{\frac{p}{p-q}}\right.
$$
\begin{equation}\label{eq3.20-1}
\left.\times\left(\int\limits_0^z K_i^{p'}(z, s)v^{p'}(s)ds\right)^{\frac{p(q-1)}{p-q}} d\left(\int\limits_0^z K_{i}^{p'}(z, s)v^{p'}(s)ds\right)\right)^{\frac{p-q}{pq}}\ll \|\mathbf{K^+}\|_{p\rightarrow q}.
\end{equation}
Since
$$
\int\limits_z^\infty \int\limits_{[t,\infty]} K_{n, i}(x, t)\left(\int\limits_t^x K_{n, i}(x, \tau)\varphi(\tau)d\tau\right)^{q-1}d\mu(x)\varphi(t)dt
$$
$$
=\int\limits_{[z,\infty]}\int\limits_z^xK_{n, i}(x, t)\varphi(t)\left(\int\limits_t^x K_{n, i}(x, \tau)\varphi(\tau)d\tau\right)^{q-1}dtd\mu(x)\approx\int\limits_{[z,\infty]}\left(\int\limits_z^xK_{n,i}(x,t)\varphi(t)dt\right)^qd\mu(x),
$$
then from \eqref{eq3.20-1} we have
$$
\sup\limits_{\varphi}\left(\int\limits_{(0,\infty]}\left(\int\limits_{[z,\infty]}\left(\int\limits_z^xK_{n,i}(x,t)\varphi(t)dt\right)^qd\mu(x)\right)^{\frac{p}{p-q}}\right.
$$
$$
\left.\times\left(\int\limits_0^z K_i^{p'}(z, s)v^{p'}(s)ds\right)^{\frac{p(q-1)}{p-q}} d\left(\int\limits_0^z K_{i}^{p'}(z, s)v^{p'}(s)ds\right)\right)^{\frac{p-q}{pq}}\ll \|\mathbf{K^+}\|_{p\rightarrow q}.
$$
Let $\tau\in I$. Assume that $\varphi(t)=\delta(t-\tau)$. Then
$$
\sup\limits_{\tau\in I}\left(\int\limits_{(0,\infty]}\left(\int\limits_{[z,\infty]} \left(\int\limits_0^\infty \chi_{(z, x)}(t) K_{n, i}(x, t)\delta(t-\tau)dt\right)^qd\mu(x)\right)^{\frac{p}{p-q}}\left(\int\limits_0^x K_{i}^{p'}(x, s)v^{p'}(s)ds\right)^{\frac{p(q-1)}{p-q}}\right.
$$
$$
\left. \times d\left(\int\limits_0^z K_{i}^{p'}(z, s)v^{p'}(s)ds\right)\right)^{\frac{p-q}{pq}}\ll \|\mathbf{K^+}\|_{p\rightarrow q},
$$
which, basing on the non-increase of the function $K_{n, i}(x, t)$ in the second argument, implies that
$$
B^+_{n, i}=\left(\int\limits_{(0,\infty]}\left(\int\limits_{[z,\infty]}  K_{n, i}^q(x, z)d\mu(x)\right)^{\frac{p}{p-q}}\left(\int\limits_0^z K_i^{p'}(z, s)v^{p'}(s)ds\right)^{\frac{p(q-1)}{p-q}}\right.
$$
$$
\left.\times d\left(\int\limits_0^z K_{i}^{p'}(z, s)v^{p'}(s)ds\right)\right)^{\frac{p-q}{pq}}\ll \|\mathbf{K^+}\|_{p\rightarrow q},
$$
i.e.,
$$
B^+_{n, i}\ll  \|\mathbf{K^+}\|_{p\rightarrow q},~ i=1,2,...,n-1,
$$ holds.
The latter, together with  \eqref{eq3.16} and \eqref{eq3.18}, gives that
\begin{equation}\label{eq3.22}
B^+_n=\max\limits_{0\leq i\leq n}B^+_{n,i}\ll \|\mathbf{K^+}\|_{p\rightarrow q}.
\end{equation}

{\it Sufficiency.} Let $B^+_n<\infty$ and $K(\cdot,\cdot)\equiv K_n(\cdot, \cdot)\in \mathcal{O}_n^{+}$. Let $0\leq f\in L_p(I)$ have compact support. Next, we proceed in the same way as in the proof of the sufficient part of Theorem~\ref{2.1}. Here we take the function $F_n(x)=\int\limits_0^x K_n(x, s)v(s)f(s)ds$ instead of the function $F(\cdot)$ and, as in Theorem \ref{2.1}, integers $\{k_i\}\subset \mathbb{Z},$ $(k_i<k_{i+1})$ and points $\{x_i\}\subset I$ are such that
\begin{equation}\label{eq3.3}
I=\bigcup\limits_i I_i,  ~~~ I_i=[x_i, x_{i+1}),\,\,\,\bar{I}_i=[x_i,x_{i+1}],
\end{equation}
\begin{equation}\label{eq3.4}
(h_n+1)^{k_i}\leq F_n(x)< (h_n+1)^{k_i+1}~~\text{for}~~x_i\leq x<x_{i+1},
\end{equation}
where $h_n$ is  from relation (\ref{eq3.1}).
Based on  (\ref{eq3.1}), relation \eqref{eq2.8} has the form
$$
(h_n+1)^{k_i-1}\leq (h_n+1)^{k_i-1}-h_n(h_n+1)^{k_{i-2}+1}\leq \int\limits_{x_{i-2}}^{x_i} K_n(x_i, s) v(s) f(s)ds
$$
$$
+\int\limits_0^{x_{i-2}}[K_n(x_i, s)-h_n K_n(x_{i-2}, s)]v(s)f(s)ds
$$
$$
\ll \int\limits_{x_{i-2}}^{x_i} K_n(x_i,s)v(s) f(s)ds+ K_{n, 0}(x_i, x_{i-2})\int\limits_0^{x_{i-2}}v(s)f(s)ds
$$
\begin{equation}\label{eq3.5}
+\sum\limits_{\gamma=1}^{n-1} K_{n, \gamma}(x_i, x_{i-2})\int\limits_0^{x_{i-2}} K_\gamma(x_{i-2}, s)v(s)f(s)ds.
\end{equation}
Due to  \eqref{eq3.3}, \eqref{eq3.4} and \eqref{eq3.5} we have
$$
J_n\equiv\int\limits_{[0,\infty]}(\mathbf{K_n}^+f(x))^qd\mu(x)\ll\sum\limits_i\int\limits_{\bar{I}_i} F_n^q(x)d\mu(x)\leq\sum\limits_i(h_n+1)^{q(k_i+1)}\mu(\bar{I}_i)
$$
$$
\ll \sum\limits_i(h_n+1)^{q(k_i-1)}\mu(\bar{I}_i)\ll \sum\limits_i\mu(\bar{I}_i)\left(\int\limits_{x_{i-2}}^{x_i} K_n(x_i, s)v(s)f(s)ds\right)^q
$$
$$
+\sum\limits_i\mu(\bar{I}_i) K^q_{n, 0}(x_i, x_{i-2})\left(\int\limits_{0}^{x_{i-2}} v(s)f(s)ds\right)^q
$$
\begin{equation}\label{eq3.6}
+\sum\limits_{\gamma=1}^{n-1}\sum\limits_i\mu(\bar{I}_i) K_{n, \gamma}^q(x_i, x_{i-2})\left(\int\limits_0^{x_{i-2}}K_\gamma(x_{i-2}, s)v(s)f(s)ds\right)^q=J_{n, n}+J_{n, 0}+\sum\limits_{\gamma=1}^{n-1}J_{n, \gamma}.
\end{equation}
We estimate $J_{n, n}$, $J_{n, 0}$ and $J_{n, \gamma}$, $\gamma=1,2,...,n-1$, separately. To estimate $J_{n, n}$ we use H\"{o}lder's inequality in the integral, then in the sum with the parameters $\frac{p}{q}$ and $\frac{p}{p-q}$:
$$
J_{n, n}\leq \sum\limits_i\mu(\bar{I}_i) \left(\int\limits_{x_{i-2}}^{x_{i}}K_n^{p'}(x_{i}, s)v^{p'}(s)ds\right)^{\frac{q}{p'}}\left(\int\limits_{x_{i-2}}^{x_i}f^p(t)dt\right)^{\frac{q}{p}}
$$
$$
\leq\left(\sum\limits_i \left(\mu(\bar{I}_i)\right)^{\frac{p}{p-q}}\left(\int\limits_{x_{i-2}}^{x_{i}}K_n^{p'}(x_i, s)v^{p'}(s)ds\right)^{\frac{q(p-1)}{p-q}}\right)^{\frac{p-q}{p}}
$$
\begin{equation}\label{eq3.7}
\times\left(\sum\limits_i\int\limits_{x_{i-2}}^{x_{i}} f^p(t)dt)\right)^{\frac{q}{p}}\ll \hat{B}_{n, n}\|f\|_p^q,
\end{equation}
where
$$
\hat{B}_{n, n}=\left(\sum\limits_i \left(\mu(\bar{I}_i)\right)^{\frac{p}{p-q}}\left(\int\limits_{x_{i-2}}^{x_{i}}K_n^{p'}(x_i, s)v^{p'}(s)ds\right)^{\frac{q(p-1)}{p-q}}\right)^{\frac{p-q}{p}}.
$$
Then, applying Lemmas 1 and 2 from \cite{P2}, relation $\int\limits_{0}^{x}K_n^{p'}(x, s)v^{p'}(s)ds=\int\limits_{(0,x]}d\eta_n(z)$, which follows from the condition $B^+_n<\infty$, Remark \ref{r2.1} and Fubini's theorem, we get
$$
\hat{B}_{n, n}\ll\left(\sum\limits_i \int\limits_{[x_{i}, x_{i+1}]}\left(\mu([x,\infty]) \right)^{\frac{q}{p-q}}\left(\int\limits_{0}^{x}K_n^{p'}(x, s)v^{p'}(s)ds\right)^{\frac{q(p-1)}{p-q}}d\mu(x) \right)^{\frac{p-q}{p}}
$$
$$
\ll\left(\int\limits_{[0,\infty]}\left(\mu([x,\infty]) \right)^{\frac{q}{p-q}}\left(\int\limits_{(0,x]}d\eta_n(z)\right)^{\frac{q(p-1)}{p-q}}d\mu(x)\right)^{\frac{p-q}{p}}
$$
$$
\ll\left(\int\limits_{[0,\infty]}\left(\mu([x,\infty]) \right)^{\frac{q}{p-q}}\int\limits_{(0,x]}\left(\int\limits_{(0,s]}d\eta_n(z)\right)^{\frac{p(q-1)}{p-q}}d\eta_n(s)d\mu(x)\right)^{\frac{p-q}{p}}
$$
$$
\ll\left( \int\limits_{(0,\infty]}\left(\mu([s,\infty]) \right)^{\frac{p}{p-q}}\left(\int\limits_{(0,s]}d\eta_n(z)\right)^{\frac{p(q-1)}{p-q}}d\eta_n(s)\right)^{\frac{p-q}{p}}
$$
$$
=\left( \int\limits_{(0,\infty]}\left(\mu([x,\infty]) \right)^{\frac{p}{p-q}}\left(\int\limits_{0}^{x}K_n^{p'}(x, s)v^{p'}(s)ds\right)^{\frac{p(q-1)}{p-q}}d\left(\int\limits_{0}^{x}K_n^{p'}(x, s)v^{p'}(s)ds\right)\right)^{\frac{p-q}{p}}=(B^+_{n, n})^q.
$$
The latter and  \eqref{eq3.7} give that
\begin{equation}\label{eq3.8}
J_{n, n}\ll (B^+_{n, n})^q\|f\|_p^q.
\end{equation}
Now, we estimate $J_{n, 0}$:
$$
J_{n, 0}=\sum\limits_i\mu(\bar{I}_i)K^q_{n,0}(x_i, x_{i-2})\left(\int\limits_{0}^{x_{i-2}}v(s)f(s)ds\right)^q \leq\sum\limits_i\int\limits_{[x_i,x_{i+1}]}K^q_{n,0}(x, x_{i-2})d\mu(x)
$$
\begin{equation}\label{eq3.9}
\times\left(\int\limits_{0}^{x_{i-2}}v(s)f(s)ds\right)^q=\int\limits_{[0,\infty]}\left(\int\limits_{0}^{t} v(s)f(s)ds)\right)^qd\eta_{n, 0}(t),
\end{equation}
where $d\eta_{n, 0}(t)=\sum\limits_i\int\limits_{[x_i, x_{i+1}]}K_{n, 0}^q(x, x_{i-2}) d\mu(x)\delta(t-x_{i-2})dt$. Then by Theorem 3 from \cite[\S 1.3, 1.3.3]{Maz2} we obtain
\begin{equation}\label{eq3.10}
J_{n, 0}\ll \hat{B}_{n, 0}\|f\|_p^q,
\end{equation}
where
$$
\hat{B}_{n, 0}=\left(\int\limits_0^\infty (\eta_{n, 0}([x, \infty]))^{\frac{p}{p-q}}\left(\int\limits_0^x v^{p'}(s)ds\right)^{\frac{p(q-1)}{p-q}}v^{p'}(x)dx\right)^{\frac{p-q}{p}}.
$$
Since
$$\eta_{n, 0}([x, \infty])=\sum\limits_{i: x_{i-2}\geq x}\int\limits_{[x_i, x_{i+1}]} K_{n, 0}^q(t, x_{i-2})d\mu(t)\ll \int\limits_{[x,\infty]} K_{n, 0}^q(t, x)\mu(t),$$
then $\hat{B}_{n, 0}\ll (B^+_{n, 0})^q$. Therefore, from \eqref{eq3.10} we find that
\begin{equation}\label{eq3.11}
J_{n, 0}\ll (B^+_{n, 0})^q\|f\|_p^q.
\end{equation}
Note that relations \eqref{eq3.8} and \eqref{eq3.11} hold for all $n\geq 1$.

Next, we estimate $J_{n, j}$, $j=1,2,...,n-1$:
$$
J_{n,j}\leq\sum\limits_i\int\limits_{[x_i,x_{i+1}]}K^q_{n,j}(x,x_{i-2})d\mu(x)\left(\int\limits_{0}^{x_{i-2}}K_j(x_{i-2},s)v(s)f(s)ds\right)^q
$$
\begin{equation}\label{eq3.11.1}
=\int\limits_{[0,\infty]}\left(\int\limits_0^t K_j(t, s)v(s)f(s)ds\right)^qd\eta_{n, j}(t),
\end{equation}
where $d\eta_{n, j}(t)=\sum\limits_i\int\limits_{[x_i,x_{i+1}]}K_{n, j}^q(x, x_{i-2})d\mu(x) \delta(t-x_{i-2})dt.$
Further, by the method of mathematical induction, we show that
\begin{equation}\label{eq3.11.2}
J_{n,j}\ll \max\limits_{0\leq i\leq j}(B^+_{n,i})^q\|f\|_p^q
\end{equation}
for all $j=1,2,...,n-1$.

Hence, from \eqref{eq3.8}, \eqref{eq3.11} and \eqref{eq3.11.2} we get
$$
J_n\ll (B^+_n)^q \|f\|_p^q,
$$
i.e., operator \eqref{eq2.1}  is bounded from $L_{p}(I)$ to $L_{q,\mu}(I)$ and for its norm $\|\mathbf{K^+}\|_{p\rightarrow q}$ from $L_{p}(I)$ to $L_{q,\mu}(I)$ we have the estimate
\begin{equation}\label{eq3.14}
\|\mathbf{K^+}\|_{p\rightarrow q}\ll B^+_{n}.
\end{equation}

Let $j=1$.  Since the measure $\eta_{n, j}$ is generated by a non-increasing function $\psi(z)=\sum\limits_{i:x_{i-2}\geq z}\int\limits_{[x_i,x_{i+1}]}K_{n, j}^q(x, x_{i-2})d\mu(x)$, i.e., $\int\limits_{[z,\infty]}d\eta_{n, j}(t)=\psi(z)$, $z\in I$, then it is obvious that $\eta_{n, j}\in \Psi$ for all $j=1,2,...,n-1$.
Therefore, applying Theorem \ref{2.1} to $J_{n,1}$, we deduce
\begin{equation}\label{eq3.12}
J_{n, 1}\ll \max\{\hat{B}_{n, 1, 0}, \hat{B}_{n, 1, 1}\}\|f\|_p^q,
\end{equation}
where
$$
\hat{B}_{n, 1, 0}=\left(\int\limits_0^\infty\left(\int\limits_{[z, \infty]} K_{1, 0}^q(t, z)d\eta_{n, 1}(t)\right)^{\frac{p}{p-q}}\left(\int\limits_0^zv^{p'}(s)ds\right)^{\frac{p(q-1)}{p-q}}v^{p'}(z)dz\right)^{\frac{p-q}{p}},
$$
$$
\hat{B}_{n, 1, 1}=\left(\int\limits_{(0, \infty]}\left(\eta_{n, 1}([z, \infty])\right)^{\frac{p}{p-q}} \left(\int\limits_0^z K_{1}^{p'}(z, s)v^{p'}(s)ds\right)^{\frac{p(q-1)}{p-q}} d\left(\int\limits_0^z K_1^{p'}(z, t)v^{p'}(t)dt\right)\right)^{\frac{p-q}{p}}.
$$
Using Remark \ref{r2.0}, we obtain
$$
\int\limits_{[z, \infty]}K_{1, 0}^q(t, z)d\eta_{n, 1}(t)=\sum\limits_{i: x_{i-2}\geq z}K_{1, 0}^q(x_{i-2}, z)\int\limits_{[x_i, x_{i+1}]}K_{n, 1}^q(x, x_{i-2})d\mu(x)
$$
$$
\ll\sum\limits_{i: x_{i-2}\geq z}\int\limits_{[x_i, x_{i+1}]}K^q_{n, 0}(x, z)d\mu(x)\ll \int\limits_{[z,\infty]} K_{n, 0}^q(x, z)d\mu(x).
$$
Then
$$
\int\limits_{[z, \infty]}K_{1, 0}^q(t, z)d\eta_{n, 1}(t)\ll \int\limits_{[z,\infty]} K_{n, 0}^q(x, z)d\mu(x).
$$
Replacing the obtained estimate in the expression $\hat{B}_{n, 1, 0}$, we get
\begin{equation}\label{eq3.13}
\hat{B}_{n, 1, 0}\ll (B^+_{n, 0})^q.
\end{equation}
Since $\eta_{n, 1}([z, \infty])=\sum\limits_{i: x_{i-2}\geq z}\int\limits_{[x_i, x_{i+1}]}K_{n, 1}^q(x, x_{i-2})d\mu(x)\ll \int\limits_{[z, \infty]}K_{n, 1}^q(x, z)d\mu(x)$,
then $\hat{B}_{n, 1, 1}\ll (B^+_{n, 1})^q$. Thus, taking into account \eqref{eq3.12} and \eqref{eq3.13}, we have
$$
J_{n, 1}\ll \max\{B^+_{n, 0}, B^+_{n, 1}\}^q\|f\|_p^q,
$$
i.e., relation \eqref{eq3.11.2} holds for $j=1$ for any $n>1$.

Let $1\leq k<n-1$ and relation \eqref{eq3.11.2} hold for $j=1,2,...,k$ for any $n>k$ and $\mu\in \Psi$. We show that it holds for $j=k+1$. From \eqref{eq3.11.1} we obtain
\begin{equation}\label{F3.3}
J_{n,k+1}\leq\int\limits_{[0, \infty]}\left(\int\limits_0^t K_{k+1}(t, s)v(s)f(s)ds\right)d\eta_{n, k+1}(t).
\end{equation}
Since  $\eta_{n, k+1}\in \Psi$, then from \eqref{F3.3} and \eqref{eq3.6} we have
\begin{equation}\label{F3.4}
\int\limits_{[0, \infty]}\left(\int\limits_0^t K_{k+1}(t, s)v(s)f(s)ds\right)d\eta_{n, k+1}(t)\ll J_{n,k+1,k+1}+J_{n,k+1, 0}+\sum\limits_{\gamma=1}^{k}J_{n,k+1,\gamma}.
\end{equation}
From  \eqref{eq3.8} and \eqref{eq3.11} we respectively get
\begin{equation}\label{F3.5}
 J_{n,k+1,k+1}\ll (B^+_{n,k+1,k+1})^q\|f\|^q_p,~~J_{n,k+1,0} \ll (B^+_{n,k+1,0})^q\|f\|^q_p,
\end{equation}
where
\begin{multline*}
(B^+_{n,k+1,k+1})^q=\left( \int\limits_{(0, \infty]}\left(\eta_{n, k+1}([x,\infty]) \right)^{\frac{p}{p-q}}\left(\int\limits_{0}^{x}K_{k+1}^{p'}(x, s)v^{p'}(s)ds\right)^{\frac{p(q-1)}{p-q}}\right.\\
\left.d\left(\int\limits_{0}^{x}K_{k+1}^{p'}(x, s)v^{p'}(s)ds\right)\right)^{\frac{p-q}{p}},
\end{multline*}
$$
(B^+_{n,k+1,0})^q=\left(\int\limits_0^\infty \left(\int\limits_{[x,\infty]} K_{k+1, 0}^q(t, x)\,d\eta_{n, k+1}(t)\right)^{\frac{p}{p-q}}\left(\int\limits_0^x v^{p'}(s)ds\right)^{\frac{p(q-1)}{p-q}}v^{p'}(x)dx\right)^{\frac{p-q}{p}}.
$$
Since
$$
\eta_{n, k+1}([x,\infty])=\sum\limits_{i:x_{i-2}\geq x}\int\limits_{[x_i,x_{i+1}]}K_{n, k+1}^q(t, x_{i-2})d\mu(t)\ll \int\limits_{[x,\infty]}K_{n, k+1}^q(t, x)d\mu(t),
$$  and
$$
\int\limits_{[x,\infty]} K_{k+1, 0}^q(t, x)\,d\eta_{n, k+1}(t)=\sum\limits_{i:x_{i-2}\geq x}K_{k+1, 0}^q(x_{i-2}, x)\int\limits_{[x_i,x_{i+1}]}K_{n, k+1}^q(t, x_{i-2})d\mu(t)
$$
\begin{equation}\label{F3.5+1}
=\sum\limits_{i:x_{i-2}\geq x}\int\limits_{[x_i,x_{i+1}]}K_{n, k+1}^q(t, x_{i-2})K_{k+1, 0}^q(x_{i-2}, x)d\mu(t)\ll\int\limits_{[x,\infty]}K_{n, 0}^q(t, x)d\mu(t)
\end{equation}
due to Remark  \ref{r2.0}, then
$$
(B^+_{n,k+1,k+1})^q\ll\left( \int\limits_{(0, \infty]}\left(\int\limits_{[x,\infty]}K_{n, k+1}^q(t, x)d\mu(t)\right)^{\frac{p}{p-q}}\left(\int\limits_{0}^{x}K_{k+1}^{p'}(x, s)v^{p'}(s)ds\right)^{\frac{p(q-1)}{p-q}}\right.
$$
$$
\left.\times d\left(\int\limits_{0}^{x}K_{k+1}^{p'}(x, s)v^{p'}(s)ds\right)\right)^{\frac{p-q}{p}}=(B^+_{n,k+1})^q,
$$
$$
(B^+_{n,k+1,0})^q\ll\left(\int\limits_0^\infty \left( \int\limits_{[x,\infty]}K_{n, 0}^q(t, x)d\mu(t)\right)^{\frac{p}{p-q}}\left(\int\limits_0^x v^{p'}(s)ds\right)^{\frac{p(q-1)}{p-q}}v^{p'}(x)dx\right)^{\frac{p-q}{p}}
=(B^+_{n,0})^q.
$$
Since $1\leq \gamma\leq k$ in \eqref{F3.4}, by the assumption, we have
$$
J_{n,k+1,\gamma}\ll \max\limits_{0\leq i\leq \gamma}(B^+_{n,k+1,i})^q\|f\|_p^q,
$$
where
$$
(B^+_{n,k+1,i})^q=\left(\int\limits_{(0, \infty]}\left(\int\limits_{[z,\infty]} K_{k+1, i}^q(x, z) d\eta_{n, k+1}(x)\right)^{\frac{p}{p-q}}\left(\int\limits_0^z K_i^{p'}(z, s)v^{p'}(s)ds\right)^{\frac{p(q-1)}{p-q}}\right.
$$
$$
\left.\times d\left(\int\limits_0^z K^{p'}_i (z, t) v^{p'}(t)dt\right)\right)^{\frac{p-q}{p}}
$$
$$
\ll\left(\int\limits_{(0, \infty]}\left(\int\limits_{[z, \infty]} K_{n, i}^q(x, z) d\mu(x)\right)^{\frac{p}{p-q}}\left(\int\limits_0^z K_i^{p'}(z, s)v^{p'}(s)ds\right)^{\frac{p(q-1)}{p-q}}\right.
$$
$$
\left.\times d\left(\int\limits_0^z K^{p'}_i (z, t) v^{p'}(t)dt\right)\right)^{\frac{p-q}{p}}=(B^+_{n,i})^q.
$$
In the last relation, as for \eqref{F3.5+1}, we have used the expression for $d\eta_{n, k+1}(x)$ and Remark~\ref{r2.0}. Then
\begin{equation}\label{F3.6}
J_{n,k+1,\gamma}\ll \max\limits_{0\leq i\leq \gamma}(B^+_{n,i})^q\|f\|_p^q,~~\gamma=1,2,...,k.
\end{equation}
Now, from \eqref{F3.3}, \eqref{F3.4}, \eqref{F3.5} and \eqref{F3.6} we deduce that \eqref{eq3.11.2} holds for $j=k+1$.

Thus, we have proved that \eqref{eq3.11.2} holds. Hence, \eqref{eq3.14} holds, which, together with \eqref{eq3.22}, gives that
$\|\mathbf{K^+}\|_{p\rightarrow q}\approx B^+_{n}$. Let us show that relation \eqref{eq3.2} holds. Let  $i\in\{0,1,...,n\}$.

The function  $\int\limits_{[z,\infty]} K_{n, i}^q(x, z) d\mu(x)$ does not increase,  therefore, by Remark \ref{r2.2}, there exists a Borel measure $\theta_{n,i}$ such that $\int\limits_{[z,\infty]} K_{n, i}^q(x, z) d\mu(x)=\int\limits_{[s,\infty]}d\theta_{n,i}(z)$. By Lemma 2 from \cite{P2}, we have
$$
\left(\int\limits_{[z,\infty]} K_{n, i}^q(x, z) d\mu(x)\right)^{\frac{p}{p-q}}=\left(\int\limits_{[s,\infty]}d\theta_{n,i}(z)\right)^{\frac{p}{p-q}}\approx \int\limits_{[s,\infty]}\left(\int\limits_{[x,\infty]}d\theta_{n,i}(z)\right)^{\frac{q}{p-q}}d\theta_{n,i}(x).
$$
Replacing this relation into the expression $B^+_{n,i}$ and using Fubini's theorem, we get
\begin{multline*}
(B^+_{n,i})^{\frac{pq}{p-q}}\approx\int\limits_{[0, \infty]}\left(\int\limits_0^z K_i^{p'}(z, s)v^{p'}(s)ds\right)^{\frac{q(p-1)}{p-q}}\left(\int\limits_{[z,\infty]} K_{n, i}^q(x, z) d\mu(x)\right)^{\frac{q}{p-q}}\\
\times d\left(-\int\limits_{[z,\infty]} K_{n, i}^q(x, z) d\mu(x)\right).
\end{multline*}
Hence, relation \eqref{eq3.2} holds.  The proof of Theorem \ref{th3.1} is complete.
\end{proof}

Now, we consider operator \eqref{eq2.2} from $L_{p,\mu}(I)$ to $L_{q}(I)$ when its kernel $K(\cdot,\cdot)\in \mathcal{O}_n^{-}$, $n\geq1$. In \eqref{1.5} we assume that $K_0(\cdot,\cdot)\equiv 1$.

We can similarly prove the following statement.

\begin{theorem}\label{th3.2} Let $1<q<p<\infty$,  $\mu\in\Psi$ and $K(\cdot,\cdot)\equiv K_n(\cdot, \cdot)\in \mathcal{O}_n^{-}$, $n>1$. Then operator \eqref{eq2.2} is bounded from $L_{p,\mu}(I)$ to $L_{q}(I)$ if and only if $B^-_n=\max\limits_{0\leq i\leq n}B^-_{i,n}<\infty$. Moreover, $\|\mathbf{K^-}\|_{p\rightarrow q}\approx B^-_n$, where $\|\mathbf{K^-}\|_{p\rightarrow q}$ is the norm of operator \eqref{eq2.2}  from $L_{p,\mu}(I)$ to $L_{q}(I)$ and
\begin{multline*}
B^-_{i, n}=\left(\int\limits_{[0, \infty]}\left(\int\limits_0^z K_{i, n}^q(z, s) v^q(s)ds\right)^{\frac{p}{p-q}}\left(\int\limits_{[z, \infty]} K_i^{p'}(x, z)
 d\mu(x)\right)^{\frac{p(q-1)}{p-q}}\right.\\
\left.\times d\left(-\int\limits_{[z, \infty]} K_i^{p'}(x, z)d\mu(x)\right)\right)^{\frac{p-q}{pq}}
\end{multline*}
\begin{multline*}
\approx\left(\int\limits_{(0, \infty]}\left(\int\limits_{[z, \infty]} K_i^{p'}(x, z)d\mu(x)\right)^{\frac{q(p-1)}{p-q}}\left(\int\limits_0^z K_{i, n}^q(z, s) v^q(s)ds\right)^{\frac{q}{p-q}}\right.\\
\left.\times d\left(\int\limits_0^z K_{i, n}^q(z, s) v^q(s)ds\right)\right)^{\frac{p-q}{pq}},~i=0,1,...,n.
\end{multline*}
\end{theorem}

Now, we consider operator \eqref{eq2.2} from $L_{p,\mu}(I)$ to $L_q(I)$ and operator \eqref{eq2.1} from $L_{p}(I)$ to $L_{q,\mu}(I)$ for $1<q<p<\infty$ when their kernels belong to the classes $ \mathcal{O}_n^{+}$ and
$\mathcal{O}_n^{-}$, $n>1$, respectively.  In view of mutual duality of operators \eqref{eq2.2} and \eqref{eq2.1}, from Theorems \ref{th3.2} and \ref{th3.1} we respective have the following two theorems.

\begin{theorem}\label{th3.3} Let $1<q<p<\infty$, $\mu\in\Psi$ and $K(\cdot,\cdot)\equiv K_n(\cdot, \cdot)\in \mathcal{O}_n^{-}$, $n\geq1$. Then operator \eqref{eq2.1} is bounded from $L_{p}(I)$ to $L_{q,\mu}(I)$ if and only if $\bar{B}^+_n=\max\limits_{0\leq i\leq n}\bar{B}^+_{i,n}<\infty$. Moreover, $\|\mathbf{K}^+\|_{p\rightarrow q}\approx \bar{B}_n^+$, where $\|\mathbf{K}^+\|_{p\rightarrow q}$ is the norm of operator \eqref{eq2.1} from $L_{p}(I)$ to $L_{q,\mu}(I)$ and
\begin{multline*}
\bar{B}^+_{i, n}=\left(\int\limits_{[0, \infty]}\left(\int\limits_0^z K_{i, n}^{p'}(z, s) v^{p'}(s)ds\right)^{\frac{q(p-1)}{p-q}}\left(\int\limits_{[z, \infty]} K_i^{q}(x, z)d\mu(x)\right)^{\frac{q}{p-q}}\right.\\
\left.\times d\left(-\int\limits_{[z, \infty]} K_i^{q}(x, z)d\mu(x)\right)\right)^{\frac{p-q}{pq}}
\end{multline*}
\begin{multline*}
\approx\left(\int\limits_{(0, \infty]}\left(\int\limits_{[z, \infty]} K_i^{q}(x, z)d\mu(x)\right)^{\frac{p}{p-q}}\left(\int\limits_0^z K_{i, n}^{p'}(z, s) v^{p'}(s)ds\right)^{\frac{p(q-1)}{p-q}}\right.\\
\left.\times d\left(\int\limits_0^z K_{i, n}^{p'}(z, s) v^{p'}(s)ds\right)\right)^{\frac{p-q}{pq}},~i=0,1,...,n.
\end{multline*}

\end{theorem}

\begin{theorem}\label{th3.4} Let $1<q<p<\infty$,  $\mu\in\Psi$ and $K(\cdot,\cdot)\equiv K_n(\cdot, \cdot)\in \mathcal{O}_n^{+}$, $n>1$. Then operator \eqref{eq2.2} is bounded from $L_{p,\mu}(I)$ to $L_{q}(I)$ if and only if $\bar{B}^-_n=\max\limits_{0\leq i\leq n}\bar{B}^-_{n,i}<\infty$. Moreover, $\|\mathbf{K}^-\|_{p\rightarrow q}\approx \bar{B}_n^-$, where $\|\mathbf{K}^-\|_{p\rightarrow q}$ is the norm of operator \eqref{eq2.2} from $L_{p,\mu}(I)$ to $L_{q}(I)$ and
\begin{multline*}
\bar{B}^-_{n, i}=\left(\int\limits_{(0, \infty]}\left(\int\limits_{[z,\infty]} K_{n, i}^{p'}(x, z) d\mu(x)\right)^{\frac{q(p-1)}{p-q}}\left(\int\limits_0^z K_i^{q}(z, s)v^{q}(s)ds\right)^{\frac{q}{p-q}}\right.\\
\left.\times d\left(\int\limits_0^z K^{q}_i (z, t) v^{q}(t)dt\right)\right)^{\frac{p-q}{pq}}
\end{multline*}
\begin{multline*}
\approx\left(\int\limits_{[0, \infty]}\left(\int\limits_0^z K_i^{q}(z, s)v^{q}(s)ds\right)^{\frac{p}{p-q}}\left(\int\limits_{[z,\infty]} K_{n, i}^{p'}(x, z) d\mu(x)\right)^{\frac{p(q-1)}{p-q}}\right.\\
\left.\times d\left(-\int\limits_{[z,\infty]} K_{n, i}^{p'}(x, z) d\mu(x)\right)\right)^{\frac{p-q}{pq}},~i=1,2,...,n-1.
\end{multline*}

\end{theorem}

Let $\mu\in \Psi$.  We consider the integral operator
\begin{equation}\label{eq4.1}
\mathbb{K^-}g(s)=\int\limits_s^\infty K(x, s)u(x)g(x)dx,~s>0,
\end{equation}
acting from $L_{p}(I)$ to $L_{q, \mu}(I)$, and the integral operator
\begin{equation}\label{eq4.2}
\mathbb{K^+}f(x)=\int\limits_{[0,x]}u(x) K(x, s)f(s)d\mu(s), ~x>0,
\end{equation}
acting from  $L_{p, \mu}(I)$ to $L_{q}(I)$.

Proceeding as above, we obtain analogues of Theorems~\ref{th3.1}--\ref{th3.4}.

\begin{theorem}\label{th3.5} Let $1<q<p<\infty$,  $\mu\in\Psi$ and $K(\cdot,\cdot)\equiv K_n(\cdot, \cdot)\in \mathcal{O}_n^{-}$, $n\geq1$. Then operator \eqref{eq4.1} is bounded from $L_{p}(I)$ to $L_{q,\mu}(I)$ if and only if $\hat{B}^-_n=\max\limits_{0\leq i\leq n}\hat{B}^-_{i, n}<\infty$. Moreover, $\|\mathbb{K^-}\|_{p\rightarrow q}\approx \hat{B}^-_n$, where $\|\mathbb{K^-}\|_{p\rightarrow q}$ is the norm of operator \eqref{eq4.1} from $L_{p}(I)$ to $L_{q,\mu}(I)$ and
\begin{multline*}
\hat{B}^-_{i, n}=\left(\int\limits_{[0,\infty)}\left(\int\limits_{[0,z]} K_{i, n}^q(z, x) d\mu(x)\right)^{\frac{p}{p-q}}\left(\int\limits_z^\infty K_i^{p'}(s, z)u^{p'}(s)ds\right)^{\frac{p(q-1)}{p-q}}\right.\\
\left.\times d\left(-\int\limits_z^\infty K^{p'}_i (t, z) u^{p'}(t)dt\right)\right)^{\frac{p-q}{pq}},~i=0,1,...,n.
\end{multline*}
\begin{multline*}
\approx
\left(\int\limits_{[0,\infty]}\left(\int\limits_z^\infty K_i^{p'}(s, z)u^{p'}(s)ds\right)^{\frac{q(p-1)}{p-q}}\left(\int\limits_{[0,z]} K_{i, n}^q(z, x)
 d\mu(x)\right)^{\frac{q}{p-q}}\right.\\
 \left.\times d\left(\int\limits_{[0,z]} K_{i, n}^q(z, x) d\mu(x)\right)\right)^{\frac{p-q}{pq}},~i=0,1,...,n.
\end{multline*}
\end{theorem}

\begin{theorem}\label{th3.6} Let $1<q<p<\infty$,  $\mu\in\Psi$ and $K(\cdot,\cdot)\equiv K_n(\cdot, \cdot)\in \mathcal{O}_n^{+}$, $n\geq1$. Then operator \eqref{eq4.2} is bounded from $L_{p,\mu}(I)$ to $L_{q}(I)$ if and only if $\hat{B}^+_n=\max\limits_{0\leq i\leq n}\hat{B}^+_{n,i}<\infty$. Moreover, $\|\mathbb{K^+}\|_{p\rightarrow q}\approx \hat{B}^+_n$, where $\|\mathbb{K^+}\|_{p\rightarrow q}$ is the norm of operator \eqref{eq4.2}  from $L_{p,\mu}(I)$ to $L_{q}(I)$ and
\begin{multline*}
\hat{B}^+_{n, i}=\left(\int\limits_{[0, \infty]}\left(\int\limits_z^\infty K_{n, i}^q(s, z) u^q(s)ds\right)^{\frac{p}{p-q}}\left(\int\limits_{[0, z]} K_i^{p'}(z, x)
 d\mu(x)\right)^{\frac{p(q-1)}{p-q}}\right.\\
\left.\times d\left(\int\limits_{[0, z]} K_i^{p'}(z, x)d\mu(x)\right)\right)^{\frac{p-q}{pq}}
\end{multline*}
\begin{multline*}
\approx\left(\int\limits_{[0, \infty)}\left(\int\limits_{[0, z]} K_i^{p'}(z, x)d\mu(x)\right)^{\frac{q(p-1)}{p-q}}\left(\int\limits_z^\infty K_{n, i}^q(s, z) u^q(s)ds\right)^{\frac{q}{p-q}}\right.\\
\left.\times d\left(-\int\limits_z^\infty K_{n, i}^q(s, z) u^q(s)ds\right)\right)^{\frac{p-q}{pq}},~i=0,1,...,n.
\end{multline*}
\end{theorem}

\begin{theorem}\label{th3.7} Let $1<q<p<\infty$, $\mu\in\Psi$ and $K(\cdot,\cdot)\equiv K_n(\cdot, \cdot)\in \mathcal{O}_n^{+}$, $n\geq1$. Then operator \eqref{eq4.1} is bounded from $L_{p}(I)$ to $L_{q,\mu}(I)$ if and only if $\tilde{B}^-_n=\max\limits_{0\leq i\leq n}\tilde{B}^-_{n,i}<\infty$. Moreover, $\|\mathbb{K}^-\|_{p\rightarrow q}\approx \tilde{B}_n^-$, where $\|\mathbb{K}^-\|_{p\rightarrow q}$ is the norm of operator \eqref{eq4.1} from $L_{p}(I)$ to $L_{q,\mu}(I)$ and
\begin{multline*}
\tilde{B}^-_{n, i}=\left(\int\limits_{[0, \infty]}\left(\int\limits_z^\infty K_{n, i}^{p'}(s, z) u^{p'}(s)ds\right)^{\frac{q(p-1)}{p-q}}\left(\int\limits_{[0, z]} K_i^{q}(z, x)d\mu(x)\right)^{\frac{q}{p-q}}\right.\\
\left.\times d\left(\int\limits_{[0, z]} K_i^{q}(z, x)d\mu(x)\right)\right)^{\frac{p-q}{pq}}
\end{multline*}
\begin{multline*}
\approx\left(\int\limits_{[0, \infty)}\left(\int\limits_{[0, z]} K_i^{q}(z, x)d\mu(x)\right)^{\frac{p}{p-q}}\left(\int\limits_z^\infty K_{n, i}^{p'}(s, z) u^{p'}(s)ds\right)^{\frac{p(q-1)}{p-q}}\right.\\
\left.\times d\left(-\int\limits_z^\infty K_{n, i}^{p'}(s, z) u^{p'}(s)ds\right)\right)^{\frac{p-q}{pq}},~i=0,1,...,n.
\end{multline*}

\end{theorem}

\begin{theorem}\label{th3.8} Let $1<q<p<\infty$,  $\mu\in\Psi$ and $K(\cdot,\cdot)\equiv K_n(\cdot, \cdot)\in \mathcal{O}_n^{-}$, $n>1$. Then operator \eqref{eq4.2} is bounded from $L_{p,\mu}(I)$ to $L_{q}(I)$ if and only if $\tilde{B}^+_n=\max\limits_{0\leq i\leq n}\tilde{B}^+_{n,i}<\infty$. Moreover, $\|\mathbb{K}^+\|_{p\rightarrow q}\approx \tilde{B}_n^+$, where $\|\mathbb{K}^+\|_{p\rightarrow q}$ is the norm of operator \eqref{eq4.2} from $L_{p,\mu}(I)$ to $L_{q}(I)$ and
\begin{multline*}
\tilde{B}^+_{i, n}=\left(\int\limits_{[0, \infty)}\left(\int\limits_{[0,z]} K_{i, n}^{p'}(z, x) d\mu(x)\right)^{\frac{q(p-1)}{p-q}}\left(\int\limits_z^\infty K_i^{q}(s, z)u^{q}(s)ds\right)^{\frac{q}{p-q}}\right.\\
\left.\times d\left(-\int\limits_z^\infty K^{q}_i (t, z) u^{q}(t)dt\right)\right)^{\frac{p-q}{pq}}
\end{multline*}
\begin{multline*}
\approx\left(\int\limits_{[0, \infty]}\left(\int\limits_z^\infty K_i^{q}(s, z)u^{q}(s)ds\right)^{\frac{p}{p-q}}\left(\int\limits_{[0,z]} K_{i, n}^{p'}(z, x) d\mu(x)\right)^{\frac{p(q-1)}{p-q}}\right.\\
\left.\times d\left(\int\limits_{[0,z]} K_{i, n}^{p'}(z, x) d\mu(x)\right)\right)^{\frac{p-q}{pq}},~i=1,2,...,n-1.
\end{multline*}

\end{theorem}

Let the measure $\mu\in \Psi$ be absolutely continuous with respect to the Lebesgue measure with density  $u^q$, where $u$ is a non-negative function such that $u\in L_q^{log}(I)$. Then
$\int\limits_0^\infty\left(\mathbf{K^+}f(x)\right)^qd\mu(x)=\int\limits_0^\infty \left(\mathcal{K^+}f(x)\right)^qdx$, $f\geq0$,
and the boundedness  of operator \eqref{eq2.1}  from  $L_{p}(I)$ to $L_{q,\mu}(I)$ coincides with the boundedness  of operator \eqref{1.1} from  $L_{p}(I)$ to $L_{q}(I)$ with the norm $\|\mathbf{K^+}\|_{p\to q}=\|\mathcal{K^+}\|_{p\to q}$. Therefore, from Theorems  \ref{th3.1} and \ref{th3.3} we respectively deduce the following statements.

\begin{theorem}\label{th3.9} Let $1<q<p<\infty$  and $K(\cdot,\cdot)\equiv K_n(\cdot, \cdot)\in \mathcal{O}_n^{+}$, $n\geq1$. Then operator \eqref{1.1} is bounded from $L_{p}(I)$ to $L_{q}(I)$ if and only if $\mathcal{B}^+_n=\max\limits_{0\leq i\leq n}\mathcal{B}^+_{n,i}<\infty$. Moreover, $\|\mathcal{K}^+\|_{p\rightarrow q}\approx \mathcal{B}_n^+$, where $\|\mathcal{K}^+\|_{p\rightarrow q}$ is the norm of operator \eqref{1.1}  from $L_{p}(I)$ to $L_{q}(I)$  and
\begin{multline*}
\mathcal{B}^+_{n, i}=\left(\int\limits_{(0, \infty]}\left(\int\limits_z^\infty K_{n, i}^q(x, z)u^q(x)dx\right)^{\frac{p}{p-q}}\left(\int\limits_0^z K_i^{p'}(z, s)v^{p'}(s)ds\right)^{\frac{p(q-1)}{p-q}}\right.\\
\left.\times d\left(\int\limits_0^z K^{p'}_i (z, t) v^{p'}(t)dt\right)\right)^{\frac{p-q}{pq}}
\end{multline*}
\begin{multline*}
\approx\left(\int\limits_{[0, \infty)}\left(\int\limits_0^z K_i^{p'}(z, s)v^{p'}(s)ds\right)^{\frac{q(p-1)}{p-q}}\left(\int\limits_z^\infty K_{n, i}^q(x, z)u^q(x)dx\right)^{\frac{q}{p-q}}\right.\\
\left.\times d\left(-\int\limits_z^\infty K_{n, i}^q(x, z)u^q(x)dx\right)\right)^{\frac{p-q}{pq}},~i=0,1,...,n.
\end{multline*}

\end{theorem}

\begin{theorem}\label{th3.10} Let $1<q<p<\infty$ and $K(\cdot,\cdot)\equiv K_n(\cdot, \cdot)\in \mathcal{O}_n^{-}$, $n\geq1$. Then operator \eqref{1.1} is bounded from $L_{p}(I)$ to $L_{q}(I)$ if and only if $\mathbb{B}^+_n=\max\limits_{0\leq i\leq n}\mathbb{B}^+_{i,n}<\infty$. Moreover, $\|\mathcal{K}^+\|_{p\rightarrow q}\approx \mathbb{B}_n^+$, where $\|\mathcal{K}^+\|_{p\rightarrow q}$ is the norm of operator \eqref{1.1} from $L_{p}(I)$ to $L_{q}(I)$ and
\begin{multline*}
\mathbb{B}^+_{i, n}=\left(\int\limits_{[0, \infty)}\left(\int\limits_0^z K_{i, n}^{p'}(z, s) v^{p'}(s)ds\right)^{\frac{q(p-1)}{p-q}}\left(\int\limits_z^\infty K_i^{q}(x, z)u^q(x)dx\right)^{\frac{q}{p-q}}\right.\\
\left.\times d\left(-\int\limits_z^\infty K_i^{q}(x, z)u^q(x)dx\right)\right)^{\frac{p-q}{pq}}
\end{multline*}
\begin{multline*}
\approx\left(\int\limits_{(0, \infty]}\left(\int\limits_z^\infty K_i^{q}(x, z)u^q(x)dx\right)^{\frac{p}{p-q}}\left(\int\limits_0^z K_{i, n}^{p'}(z, s) v^{p'}(s)ds\right)^{\frac{p(q-1)}{p-q}}\right.\\
\left.\times d\left(\int\limits_0^z K_{i, n}^{p'}(z, s) v^{p'}(s)ds\right)\right)^{\frac{p-q}{pq}},~i=0,2,...,n.
\end{multline*}

\end{theorem}

Now, we establish a criterion for the boundedness of operator \eqref{1.2} from $L_{p}(I)$ to $L_{q}(I)$ when its kernel $K(\cdot, \cdot)\equiv K_n(\cdot, \cdot)\in\mathcal{O}_n^+$.  Using mutual duality of operators \eqref{1.1} and \eqref{1.2}, from Theorems \ref{th3.9} and \ref{th3.10} we respectively obtain the following theorems.

\begin{theorem}\label{th3.11} Let $1<q<p<\infty$ and $K(\cdot,\cdot)\equiv K_n(\cdot, \cdot)\in \mathcal{O}_n^{+}$, $n\geq 1$. Then operator \eqref{1.2} is bounded from $L_{p}(I)$ to $L_{q}(I)$ if and only if $\mathcal{B}^-_n=\max\limits_{0\leq i\leq n} \mathcal{B}^-_{n, i}<\infty$. Moreover, $\|\mathcal{K}^-\|_{p\rightarrow q}\approx \mathcal{B}^-_n$, where $\|\mathcal{K}^-\|_{p\rightarrow q}$ is the norm of operator \eqref{1.2}  from $L_{p}(I)$ to $L_{q}(I)$ and
\begin{multline*}
\mathcal{B}^-_{n, i}=\left(\int\limits_{(0, \infty]}\left(\int\limits_z^\infty K_{n, i}^{p'}(x, z) u^{p'}(x)dx\right)^{\frac{q(p-1)}{p-q}}\left(\int\limits_0^z K_i^{q}(z, s)v^{q}(s)ds\right)^{\frac{q}{p-q}}\right.\\
\left.\times d\left(\int\limits_0^z K^{q}_i (z, t) v^{q}(t)dt\right)\right)^{\frac{p-q}{pq}}
\end{multline*}
\begin{multline*}
\approx\left(\int\limits_{[0, \infty)}\left(\int\limits_0^z K_i^{q}(z, s)v^{q}(s)ds\right)^{\frac{p}{p-q}}\left(\int\limits_z^\infty K_{n, i}^{p'}(x, z) u^{p'}(x)dx\right)^{\frac{p(q-1)}{p-q}}\right.\\
\left.\times d\left(-\int\limits_z^\infty K_{n, i}^{p'}(x, z) u^{p'}(x)dx\right)\right)^{\frac{p-q}{pq}},~i=0,1,...,n.
\end{multline*}

\end{theorem}

\begin{theorem}\label{th3.12} Let $1<q<p<\infty$ and $K(\cdot,\cdot)\equiv K_n(\cdot, \cdot)\in \mathcal{O}_n^{-}$, $n\geq 1$. Then operator \eqref{1.2} is bounded from $L_{p}(I)$ to $L_{q}(I)$ if and only if $\mathbb{B}^-_n=\max\limits_{0\leq i\leq n} \mathbb{B}^-_{i, n}<\infty$. Moreover, $\|\mathcal{K}^-\|_{p\rightarrow q}\approx \mathbb{B}^-_n$, where $\|\mathcal{K}^-\|_{p\rightarrow q}$ is the norm of operator \eqref{1.2} from $L_{p}(I)$ to $L_{q}(I)$ and
\begin{multline*}
\mathbb{B}^-_{i, n}=\left(\int\limits_{[0, \infty)}\left(\int\limits_0^z K_{i, n}^{q}(z, s) v^{q}(s)ds\right)^{\frac{p}{p-q}}\left(\int\limits_z^\infty K_i^{p'}(x, z)u^{p'}(x)dx\right)^{\frac{p(q-1)}{p-q}}\right.\\
\left.\times d\left(-\int\limits_z^\infty K_i^{p'}(x, z)u^{p'}(x)dx\right)\right)^{\frac{p-q}{pq}}
\end{multline*}
\begin{multline*}
\approx\left(\int\limits_{(0, \infty]}\left(\int\limits_z^\infty K_i^{p'}(x, z)u^{p'}(x)dx\right)^{\frac{q(p-1)}{p-q}}\left(\int\limits_0^z K_{i, n}^{q}(z, s) v^{q}(s)ds\right)^{\frac{q}{p-q}}\right.\\
\left.\times d\left(\int\limits_0^z K_{i, n}^{q}(z, s) v^{q}(s)ds\right)\right)^{\frac{p-q}{pq}},~i=0,1,...,n.
\end{multline*}

\end{theorem}


\bigskip
\begin{thebibliography}{99}

\bibitem{AM} K.F. Andersen, B. Muckenhoupt,  Weighted weak type Hardy inequalities with applications to Hilbert transforms and maximal functions, Studia Math. 72:1 (1982)  9--26.

\bibitem{A} L.S. Arendarenko, Estimates for Hardy-type integral operators in weighted Lebesgue spaces, PhD thesis, Lule{\aa}, Sweden 2013.

\bibitem{AOP} L.S. Arendarenko, R. Oinarov, L.-E. Persson, On the boundedness of some classes of integral operators in weighted Lebesgue spaces, Eurasian Math. J. 3:1 (2012) 5--17.


\bibitem{BKO} A. Baiarystanov, A. Kalybay, R. Oinarov,  Oscillatory and spectral properties of fourth-order differential operator and weighted differential inequality with boundary conditions, Bound. Value Probl. 2022:78 (2022). https://doi.org/10.1186/s13661-022-01659-1.


\bibitem{BK}  S. Bloom, R. Kerman,  Weighted norm inequalities for operators of Hardy type, Proc. Amer.
Math. Soc. 113:1 (1991) 135--141.

\bibitem{B}   J.S. Bradley,  Hardy inequality with mixed norms, Canad. Math. Bull. 21 (1978) 405--408.

\bibitem{CH} D. Chutia, R. Haloi, Weighted integral inequalities for modified Integral Hardy operators, Bull. Korean Math. Soc. 59:3 (2022) 757-780.

\bibitem{GS} V. Garc\'{\i}a Garc\'{\i}a, P. Ortega Salvador, Weighted weak-type iterated and bilinear Hardy inequalities, J.
Math. Anal. Appl. (2023), 127284, doi: https://doi.org/10.1016/j.jmaa.2023.127284.


\bibitem{GMPTU} A. Gogatishvili, Z. Mihula, L. Pick, H. Tur\v{c}inov\'{a}, T. \"{U}nver,  Weighted inequalities for a superposition of the Copson operator and the Hardy operator, J. Fourier Anal. Appl. 28:24 (2022). https://doi.org/10.1007/s00041-022-09918-6.

\bibitem{GM} A. Gogatishvili, R. Mustafayev, Weighted iterated Hardy-type inequalities. Math. Inequal. Appl. 20:3 (2017) 683--728.

\bibitem{KO} A. Kalybay, R. Oinarov, On weighted inequalities for a class of quasilinear integral operators, Banach J. Math. Anal. 17:3 (2023). https://doi.org/10.1007/s43037-022-00226-1.

\bibitem{KOS1}  A. Kalybay, R. Oinarov, Ya. Sultanaev, Weighted differential inequality and oscillatory properties of fourth order differential equations, J. Inequal. Appl. 2021:199 (2021). https://doi.org/10.1186/s13660-021-02731-7

\bibitem{KOS2} A. Kalybay, R. Oinarov, Ya. Sultanaev, Oscillation and spectral properties
of some classes of higher order differential operators and weighted $n$th order differential inequalities, Electron. J. Qual. Theory Differ. Equ. 3 (2021). https://doi.org/10.14232/ejqtde.2021.1.3

\bibitem{K} V.M. Kokilashvili, On Hardy's inequality in
weighted spaces, Soobscenija Akademii Nauk Gruzinskoj SSR 1:96 (1979)
37--40 (in Russian).



\bibitem{Maz1}  V.G. Maz'ya, On $(p,l)$-capacity, inbedding theorems, and the spectrum of a selfadjoint elliptic operator, Mathematics of the USSR--Izvestiya 7:2 (1973) 357--387.

\bibitem{Maz2} V.G. Maz'ya,  Sobolev spaces, Leningrad University Press, Leningrad, 1984.

\bibitem{M} B. Muckenhoupt,  Hardy's inequality with weights, Studia Math. XLIV:1 (1972) 31--38.

\bibitem{O1}  R. Oinarov, Weighted inequalities for a class of integral operators, Doklady Mathematics 44:1 (1992) 291--293.


\bibitem{O3} R. Oinarov, Boundedness and compactness of Volterra type integral operators, Siberian Math. J. 48:5 (2007) 884--896.

\bibitem{O4} R. Oinarov, Boundedness and compactness in weighted Lebesgue spaces of integral operators with variable integration limits,  Siberian Math. J. 52:6 (2011) 1042--1055.


\bibitem{P2} D.V. Prokhorov,  Hardy's inequality with three measures, Proc. Steklov Inst. Math. 255 (2006) 221--233.

\bibitem{P3} D.V. Prokhorov, On a class of weighted inequalities containing quasilinear operators, Proc. Steklov Inst. Math. 293 (2016), 272--287.

\bibitem{R} H.L. Royden,  Real analysis, Third edition,  Macmillan Publishing Company, New York, 1988.


\bibitem{S} G. Sinnamon, Weighted Hardy and Opial-type inequalities, J. Math. Anal. Appl.
160:2 (1991) 434--445.

\bibitem{SS} G. Sinnamon, V.D. Stepanov, The weighted Hardy inequalities: new proofs and the case $p=1$, J. Lond. Math. Soc. 2:54 (1996) 89--101.

\bibitem{S1}  V.D. Stepanov, On one weighted inequality of Hardy type for higher derivatives, Proc. Steklov Inst. Math. 187 (1990) 205--220.

\bibitem{S2}  V.D. Stepanov, Weighted inequalities of Hardy type for Riemann--Liouville
fractional integrals, Siberian Math. J. 3:31 (1990) 513--522.

\bibitem{S3}  V.D. Stepanov, Two--weighted estimates of Riemann--Liouville
integrals, Mathematics of the USSR--Izvestiya 36 (1991) 669--681.

\bibitem{StS1}  V.D. Stepanov,  G.E. Shambilova, On weighted iterated Hardy-type operators, Anal. Math. 44:2 (2018)  273--283.

\bibitem{StS2} V.D. Stepanov,  G.E. Shambilova, On iterated and bilinear integral Hardy-type operators, Math. Inequal. Appl. 22:4 (2019),  1505--1533.

\bibitem{SYL} Q. Sun, X. Yu, H. Li, The supremum-involving Hardy-type operators on Lorentz-type spaces, Port. Math. 77:1 (2020) 1--29.

\end{thebibliography}
\end{document}